\documentclass[a4paper]{amsart}

\usepackage[all]{xy}
\usepackage{amsmath}
\usepackage{amssymb}
\usepackage{amsthm}
\usepackage{latexsym}
\usepackage{enumerate}
\usepackage{prettyref}
\usepackage{hyperref}
\usepackage{bm}
\usepackage{mathrsfs}
\usepackage[lite]{amsrefs}
\usepackage{color}
\usepackage[margin=1in]{geometry}

\usepackage{graphicx}

\newtheorem{theorem}{Theorem}[section]
\newrefformat{thm}{\hyperref[{#1}]{Theorem~\ref*{#1}}}
\newtheorem{definition}[theorem]{Definition}
\newrefformat{def}{\hyperref[{#1}]{Definition~\ref*{#1}}}
\newtheorem{lemma}[theorem]{Lemma}
\newrefformat{lem}{\hyperref[{#1}]{Lemma~\ref*{#1}}}
\newtheorem{proposition}[theorem]{Proposition}
\newrefformat{prop}{\hyperref[{#1}]{Proposition~\ref*{#1}}}
\newtheorem{corollary}[theorem]{Corollary}
\newrefformat{cor}{\hyperref[{#1}]{Corollary~\ref*{#1}}}
\newtheorem{remark}[theorem]{Remark}
\newrefformat{rem}{\hyperref[{#1}]{Remark~\ref*{#1}}}
\newtheorem{outcome}[theorem]{Outcome}
\newrefformat{rem}{\hyperref[{#1}]{Outcome~\ref*{#1}}}
{
\newtheorem{examplecore}[theorem]{Example}}
\newrefformat{ex}{\hyperref[{#1}]{Example~\ref*{#1}}}

\newenvironment{proofof}[1]{\vspace{.2cm}\noindent\textsc{Proof of
    #1:}}{\hspace*{\fill} $\blacksquare$\par\vspace{.1cm}}
\newenvironment{example}{\begin{examplecore}}{\hspace*{\fill}
$\square$\par\vspace{.1cm}\end{examplecore}}

\newcommand*{\FT}{\widehat{\operatorname{H}}^\bullet}
\newcommand{\Z}{{\mathbb{Z}}}

\newcommand{\F}{{\mathbb{F}}}
\newcommand{\rationals}{{\mathbb{Q}}}
\newcommand{\ringO}{\mathcal{O}}

\newcommand{\op}{\operatorname}

             \newcommand{\edgegraph}{\includegraphics[width=3.78mm, height=1.8mm]{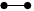}}
             \newcommand{\circlegraph}{\includegraphics[width=3.78mm, height=2.79mm]{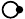}}

\begin{document}

\title[Farrell--Tate cohomology of GL$_3(\mathcal{O}_{\mathbb{Q}(\sqrt{-m})})$]{
On Farrell--Tate cohomology of \texorpdfstring{${\rm GL}_3$}{GL3} \\ over rings of quadratic integers}

\author{Bui Anh Tuan, Alexander D. Rahm and Matthias Wendt}

\date{\today}

\address{Bui Anh Tuan, Faculty of Mathematics and Computer Science, University of Science, VNU-HCM, 227 Nguyen Van Cu Str., Dist. 5, Ho Chi Minh City, Vietnam}
\email{batuan@hcmus.edu.vn}

\address{Alexander D. Rahm, Laboratoire de math\'ematiques GAATI,
Universit\'e de la Polyn\'esie fran\c{c}aise,
BP 6570,
98702 Faaa,
French Polynesia}
\email{Alexander.Rahm@upf.pf}

\address{Matthias Wendt, Fachgruppe Mathematik/Informatik, Bergische Universit\"at Wuppertal, Gaussstrasse 20, 42119 Wuppertal, Germany}
\email{wendt@math.uni-wuppertal.de}


\begin{abstract}
The goal of the present paper is to push forward the frontiers of computations on mod $\ell$ Farrell--Tate cohomology for arithmetic groups. We deal with $\ell$-rank 1 cases different from $\op{PSL}_2$. The conjugacy classification of cyclic subgroups of order $\ell$ is reduced to the classification of modules of $\op{C}_\ell$-group rings over suitable rings of integers which are principal ideal domains, generalizing an old result of Reiner. As an example of the number-theoretic input required for the Farrell--Tate cohomology computations, we discuss in detail the homological torsion in $\op{PGL}_3$ over principal ideal rings of quadratic integers,
accompanied by machine computations in the imaginary quadratic case. 
\end{abstract}

\maketitle
\setcounter{tocdepth}{1}
\tableofcontents

\section{Introduction}

In the present paper, we investigate the mod $\ell$ Farrell--Tate cohomology of some arithmetic groups $\Gamma$ of $\ell$-rank 1. This means that all elementary abelian $\ell$-groups in $\Gamma$ are in fact cyclic. Via Brown's formula for Farrell--Tate cohomology, we only need to know the number of conjugacy classes of cyclic subgroups of $\Gamma$ of order $\ell$ and the Farrell--Tate cohomology of their normalizers. It is well-known, cf. e.g. the previous computations of mod $\ell$ Farell--Tate cohomology for $\op{SL}_2(\mathcal{O}_{K,S})$ in \cite{sl2ff}, that the conjugacy classification of cyclic subgroups translates in some way into class group questions which can be handled by classical algebraic number theory. 

For the groups $\op{SL}_2(\mathcal{O}_{K,S})$ discussed in \cite{sl2ff}, the modules over the group rings $\mathcal{O}_{K,S}[\op{C}_\ell]$ corresponding to the cyclic subgroups were indecomposable. This is no longer true for the cyclic subgroups in $\op{PGL}_3(\mathcal{O}_{K,S})$, and the interaction between the indecomposable constituents of the corresponding $\mathcal{O}_{K,S}[\op{C}_\ell]$-modules has to be taken into account. On the matrix level, the situation to keep in mind is that of $(3\times 3)$-matrices with integer entries which have an upper triangular block form, one block given by a $(2\times 2)$-matrix of order $\ell$, a further diagonal entry of 1, and additional new problems that need to be handled arise from possible entries in the off-diagonal block. 

To deal with this situation, we provide a generalization of a result of Reiner~\cite{reiner:1955}, concerning the classification of $\mathbb{Z}[\op{C}_\ell]$-modules. We generalize his result to a partial classification of $\mathcal{O}_{K,S}[\op{C}_\ell]$-modules, provided that $\mathcal{O}_{K,S}$ is a principal ideal domain and the ring $\mathcal{O}_{K,S}[\zeta_\ell]$ is a Dedekind domain, cf. Theorem~\ref{thm:fullreiner} and Proposition~\ref{prop:halfreiner}. Since the conjugacy classification for cyclic subgroups translates into the isomorphism classification of such modules over the group rings, this extension of Reiner's result can be used to provide computations of numbers of conjugacy classes in a number of interesting cases. We also extend Reiner's result in a different direction: we describe the automorphism groups of the $\mathcal{O}_{K,S}[\op{C}_\ell]$-modules, and this translates into a computation of the centralizers and normalizers of cyclic subgroups in suitable arithmetic groups like $\op{PGL}_3(\mathcal{O}_{-m})$, cf. Section~\ref{sec:centralizer}. 

With these results in hand, we can then discuss a couple of sample cases and provide formulas for the mod $\ell$ Farrell--Tate cohomology. 
The cases we consider are mod 3 Farrell--Tate cohomology for rings of quadratic integers, cf. Theorem~\ref{thm:gl3_3-torsion} for the imaginary quadratic case and Proposition~\ref{prop:real} for the real quadratic case. We also discuss the 5-torsion for $\op{PGL}_3(\mathcal{O}_{\mathbb{Q}(\sqrt{5})})$, the 7-torsion for $\op{PGL}_3(\mathcal{O}_{\mathbb{Q}(\sqrt{-7})})$ and the $\ell$-torsion in $\op{PGL}_\ell(\mathbb{Z})$, cf.~Section~\ref{sec:examples}. In the imaginary quadratic case, our result describing the homological 3-torsion is the following, for a proof see Section~\ref{proofthm11}. 

\begin{theorem} \label{thm:3-torsion}
Let $A = \Z \times \Z/6$, and let $\Z/2$ act on $A$ by sign inversion to construct $A \rtimes \Z/2$. Then 
the mod $3$ Farrell-Tate cohomology of $\op{PGL}_3(\ringO_{-m})$ admits an isomorphism
$$  
\FT(\op{PGL}_3(\ringO_{-m}); \F_3) \cong \FT(A; \F_3)^{\oplus \lambda} \widetilde{\oplus} \FT(A \rtimes \Z/2; \F_3)^{\oplus \mu},
$$
where $\lambda$ respectively $\mu$ are the numbers of conjugacy classes of order-3-subgroups in  $\op{PGL}_3(\ringO_{-m})$ which do not, respectively which do, have a dihedral overgroup in $\op{PGL}_3(\ringO_{-m})$. 
\end{theorem}

Here, the cohomology rings $\FT(A; \F_3)$ and $\FT(A \rtimes \Z/2; \F_3)$ are described explicitly in Proposition~\ref{prop:ftformula} and the discussion that follows it. In the cases where $\ringO_{-m}$ is a principal ideal domain and $3\nmid m$, the numbers $\lambda$ and $\mu$ can be computed in terms of the Galois-action on the class group of $\mathbb{Q}(\sqrt{-m},\zeta_3)$ and the action of $\mathcal{O}_{\mathbb{Q}(\sqrt{-m},\zeta_3)}^\times$ on $\mathcal{O}_{-m}/(3)$, cf. Section~\ref{sec:examples}, in particular Theorem~\ref{thm:gl3_3-torsion}.
In the cases where $\mathcal{O}_{-m}$ is of very small discriminant absolute value, we can compute $\lambda$ and $\mu$ on a machine, cf. Table~\ref{results-table}, as follows.
We start with Voronoi cell complexes with $\op{GL}_3(\mathcal{O}_{\mathbb{Q}(\sqrt{-m})})$-action, constructed with Sch\"onnenbeck's software~\cites{Sebastian, BCNS}. Using a cell subdivision algorithm that was recently introduced by the authors~\cite{psl4z}, it has been possible to extract the $\ell$-torsion subcomplexes. Then for each conjugacy class of cyclic subgroups of order $\ell$, we get a connected component of the reduced $\ell$-torsion subcomplex, which determines the Farrell--Tate cohomology of its normalizer and tells us whether it contributes to the quantity $\lambda$ or $\mu$ in Theorem~\ref{thm:3-torsion}.
We will describe the machine computations in Section~\ref{sec:machine_computations}, our results being displayed in Table~\ref{results-table}. 

For the cases $m=1,2,7,11,19$, this coincides precisely with the number-theoretic formulas discussed in Section~\ref{sec:examples}. Note that the cases $m=15$ and $m=5$ are cases where the conditions of Theorem~\ref{thm:gl3_3-torsion} are not satisfied: in both cases, $\mathcal{O}_{-m}$ is not a principal ideal domain, and for $m=15$ an additional requirement for our analysis, $3\nmid m$, is violated. In these cases, we cannot yet provide general number-theoretic counts for the conjugacy classes of order-3-subgroups, but see the discussion of the case $\mathcal{O}_{-5}$ in Appendix~\ref{sec:o5}.

\emph{Structure of the paper:} A few preliminary statements are made in Section~\ref{sec:prelims}, and the translation between conjugacy classes of cyclic subgroups and modules over the group rings for the cyclic groups is explained in Section~\ref{sec:conjugacy}. We discuss a generalization of Reiner's classification of modules over group rings in Section~\ref{sec:groupring}. A description of the centralizers and normalizers of cyclic subgroups in terms of the automorphisms of the corresponding modules over the group rings is given in Section~\ref{sec:centralizer}. Finally, the sample computations are given in Section~\ref{sec:examples}, and we discuss the machine computations in Section~\ref{sec:machine_computations}.
In Appendix~\ref{sec:o5}, we discuss a sample case which exhibits the difficulties in generalizing our counting of the conjugacy classes of order-3-subgroups beyond its setting over the principal ideal domains.

\begin{table}
$$ 
\begin{array}{|l|c|c|c|c|c|c|c|}
        \hline  m & 1 & 2 & 7 & 11 & 15 & 19 & 5\\
  \hline \lambda  & 0 & 0 & 0 & 0 & 0 & 0 & 1 \\  
  \hline \mu      & 2 & 4 & 3 & 4 & 7 & 3 & 8 \\ \hline
 \end{array}
$$
\caption{Machine results on the parameters $\lambda$ and $\mu$ of Theorem \ref{thm:3-torsion}.}\label{results-table}
\end{table}

\section{Preliminaries} 
\label{sec:prelims}

\subsection{Finite subgroups}

In the present paper, we will deal with the Farrell--Tate cohomology of arithmetic groups $\Gamma$ which are of $\ell$-rank 1. This means that the maximal abelian $\ell$-subgroup of $G$ has rank 1. Note that if the arithmetic group 
$\Gamma = \op{PGL}_n(\mathcal{O}_{K,S})$ for $n\geq 2$ contains a non-trivial cyclic group $\op{C}_\ell$, then \mbox{$[K(\zeta_\ell):K]\leq n$.} 
In that case, if the $\ell$-rank is 1, then $[K(\zeta_\ell):K]>\frac{n}{2}$ (and that already excludes $\ell=2$). 
In the specific case $\Gamma=\op{PGL}_3(\mathcal{O}_{K,S})$ this means that the degree $[K(\zeta_\ell):K]$ is either $2$ or $3$. 
In the case where $K=\mathbb{Q}(\sqrt{-m})$ is an imaginary quadratic field, this restricts us to the primes $\ell=3,5,7$. Actually, the case $\ell=5$ does not appear since the quadratic subfield of $\mathbb{Q}(\zeta_5)$ is $\mathbb{Q}(\sqrt{5})$ which is real, and the case $\ell=7$ only appears for the quadratic subfield $\mathbb{Q}(\sqrt{-7})$ of $\mathbb{Q}(\zeta_7)$. 

For example, if there is a $5$-torsion element in $\op{GL}_3(\mathcal{O}_{\mathbb{Q}(\sqrt{-m})})$, then the cyclotomic field $\rationals(\zeta_5)$ must embed into the matrix algebra M$_3(\rationals(\sqrt{-m}))$ with $\zeta_5$ mapping to the $5$-torsion element. This is only possible if $[\rationals(\zeta_5,\sqrt{-m}):\rationals(\sqrt{-m})]\leq 2$, because subfields of the $3\times 3$-matrix algebra can only have degree at most~$3$. The cyclotomic field $\rationals(\zeta_5)$ has a unique quadratic subfield which is $\rationals(\sqrt{5})$ and therefore the composite $\rationals(\zeta_5,\sqrt{-m})$ has degree $8$ whenever $m$ is positive. (In examples, this can also be checked by Pari/GP~\cite{PariGP}. For instance, a minimal polynomial for the composite $\mathbb{Q}(\zeta_5,\sqrt{-5})=\mathbb{Q}(\zeta_5,\sqrt{-1})$ is given by $x^8-x^6+x^4-x^2+1$, and a minimal polynomial for the composite $\mathbb{Q}(\zeta_5,\sqrt{-3})=\mathbb{Q}(\zeta_5,\zeta_3)$ is given by $x^8-x^7+x^5-x^4+x^3-x+1$.) In particular, there can be no $5$-torsion elements in GL$_3$ over imaginary quadratic number rings.

\subsection{Farrell--Tate cohomology and Brown's formula}

To compute the relevant examples of Farrell--Tate cohomology of linear groups, we will use Brown's formula for $\ell$-rank $1$, cf.~\cite{Brown}*{Corollary X.7.4}. 
In this case, Brown's complex of elementary abelian $\ell$-subgroups of $\Gamma$ (also known as the Quillen complex) 
is in fact a disjoint union of conjugacy classes of cyclic $\ell$-subgroups of $\Gamma$, and the formula is given by
\[
\widehat{\op{H}}^\bullet(\Gamma,\mathbb{F}_\ell)\cong \prod_{[G\leq\Gamma], G\textrm{ cyclic}} \widetilde{\op{H}}^\bullet(\op{C}_{\op{\Gamma}}(G),\mathbb{F}_\ell)^{\op{N}_{\Gamma}(G)/\op{C}_{\Gamma}(G)},
\]
where the sum on the right is indexed by the conjugacy classes of finite cyclic subgroups of $\Gamma$. To evaluate the formula, we need to determine the conjugacy classes of cyclic subgroups as well as the structure of their normalizers. These questions can be translated into questions about the isomorphism classification of modules over groups rings, and the question of automorphism groups of such modules. For cyclic groups, these questions can be approached using the classical work of Reiner, cf. \cite{reiner:1955}. 

\section{Conjugacy classification of cyclic subgroups}
\label{sec:conjugacy}

In this section, we will relate the conjugacy classification of cyclic subgroups $\op{C}_\ell$ in general linear groups over $S$-integer rings $\mathcal{O}_{K,S}$ to the isomorphism classification of modules over the group rings $\mathcal{O}_{K,S}[\op{C}_\ell]$; this is a rather classical argument, cf. \cite{latimer:macduffee}. The isomorphism classification will be done in the next section, generalizing Reiner's article \cite{reiner:1955} on the isomorphism classification of modules over the integral group ring $\mathbb{Z}[\op{C}_\ell]$. 


\begin{proposition}
\label{prop:conjiso}
There is an injection from the set of conjugacy classes of embeddings $\op{C}_\ell\hookrightarrow \op{GL}_n(\mathcal{O}_{K,S})$ to the set of isomorphism classes of $\mathcal{O}_{K,S}[\op{C}_\ell]$-modules whose underlying $\mathcal{O}_{K,S}$-module is free of rank $n$. The only isomorphism class not in the image is the one where the $\op{C}_\ell$-action is trivial. 
\end{proposition}

\begin{proof}
(i) Assume we have a subgroup $\op{C}_\ell\hookrightarrow \op{GL}_n(\mathcal{O}_{K,S})$. In particular, we have a non-trivial action of $\op{C}_\ell$ on $M=\mathcal{O}_{K,S}^{\oplus n}$. We use this action to turn $M$ into an $\mathcal{O}_{K,S}[\op{C}_\ell]$-module by letting the element $[g]$ for $g\in\op{C}_\ell$ act via the embedding $\op{C}_\ell\hookrightarrow\op{GL}_n(\mathcal{O}_{K,S})$. 

(ii) Assume we have two subgroups $\phi,\phi'\colon\op{C}_\ell\hookrightarrow\op{GL}_n(\mathcal{O}_{K,S})$ which are conjugate. Then any conjugating matrix $A$ gives rise to commutative diagrams
\[
\xymatrix{
\mathcal{O}_{K,S}^{\oplus n}\ar[r]^{\phi(g)} \ar[d]_A & \mathcal{O}_{K,S}^{\oplus n} \ar[d]^A \\
\mathcal{O}_{K,S}^{\oplus n}\ar[r]_{\phi'(g)} & \mathcal{O}_{K,S}^{\oplus n} 
}
\]
showing that the two $\mathcal{O}_{K,S}[\op{C}_\ell]$-modules associated to $\phi$ and $\phi'$ are isomorphic via $A$.

(iii) Conversely, assume we have an $\mathcal{O}_{K,S}[\op{C}_\ell]$-module $M$ whose underlying $\mathcal{O}_{K,S}$-module is free of rank $n$. We choose an $\mathcal{O}_{K,S}$-basis for $M$. The representing matrices for the automorphisms $[g]$ for $g\in\op{C}_\ell$ provide an embedding $\op{C}_\ell\hookrightarrow \op{GL}_n(\mathcal{O}_{K,S})$ 
since by assumption the action of $\op{C}_\ell$ is non-trivial. Different choices of basis will give rise to subgroups which are conjugate via change-of-basis matrices.

(iv) Assume we have an isomorphism $f\colon M\cong M'$ of $\mathcal{O}_{K,S}[\op{C}_\ell]$-modules as in (iii). Then a choice of basis for $M$ induces a choice of basis for $M'$ via $f$. With these choices of bases, the modules $M$ and $M'$ give rise to the same subgroup of $\op{GL}_n(\mathcal{O}_{K,S})$. The independence-of-basis statement in (iii) implies that the subgroups associated to $M$ and $M'$ (for arbitrary choices of bases) are conjugate.
\end{proof}

\begin{remark}
It should be pointed out that there is a difference between conjugacy classes of embeddings $\op{C}_\ell\hookrightarrow\op{GL}_n(\mathcal{O}_{K,S})$ and conjugacy classes of cyclic subgroups of $\op{GL}_n(\mathcal{O}_{K,S})$ of order $\ell$. For a non-trivial automorphism $\phi\colon\op{C}_\ell\to\op{C}_\ell$ and some embedding $\iota\colon\op{C}_\ell\hookrightarrow\op{GL}_n(\mathcal{O}_{K,S})$, the two embeddings $\iota$ and $\iota\circ \phi$ are non-conjugate while obviously the images of $\iota$ and $\iota\circ\phi$ are equal as subgroups. This is similar to the difference between the conjugacy classification of order $\ell$ elements and order $\ell$ subgroups in \cite{sl2ff}. 
\end{remark}

Let $\phi\colon \op{C}_\ell\to\op{C}_\ell$ be an automorphism of the cyclic group. Then $\phi$ induces an $\mathcal{O}_{K,S}$-linear automorphism of $\mathcal{O}_{K,S}[\op{C}_\ell]$ in the obvious way. For the purposes of the next result, we call such automorphisms \emph{special}. 

\begin{corollary}
\label{cor:conjiso}
The injection of Proposition~\ref{prop:conjiso} gives rise to a bijection from the conjugacy classes of order $\ell$ subgroups in $\op{GL}_n(\mathcal{O}_{K,S})$ to the quotient of the set of isomorphism classes of $\mathcal{O}_{K,S}$-free $\mathcal{O}_{K,S}[\op{C}_\ell]$-modules by the special automorphisms. 

Under the bijection of Proposition~\ref{prop:conjiso}, the centralizer of a subgroup $\op{C}_\ell\hookrightarrow\op{GL}_n(\mathcal{O}_{K,S})$ is isomorphic to the $\mathcal{O}_{K,S}[\op{C}_\ell]$-automorphism group of the corresponding module $M$. The normalizer is isomorphic to the group of $\mathcal{O}_{K,S}$-automorphisms which are semilinear with respect to a special automorphism of $\mathcal{O}_{K,S}[\op{C}_\ell]$.
\end{corollary}

\begin{proof}
We consider a fixed subgroup (as opposed to a conjugacy class), and consider the associated module $M$, equipped with the corresponding choice of basis. Then a matrix $A$ in the centralizer of $\iota\colon \op{C}_\ell\hookrightarrow\op{GL}_n(\mathcal{O}_{K,S})$ provides commutative diagrams for all $g\in\op{C}_\ell$:
\[
\xymatrix{
\mathcal{O}_{K,S}^{\oplus n}\ar[r]^{\iota(g)} \ar[d]_A & \mathcal{O}_{K,S}^{\oplus n} \ar[d]^A \\
\mathcal{O}_{K,S}^{\oplus n}\ar[r]_{\iota(g)} & \mathcal{O}_{K,S}^{\oplus n}. 
}
\]
As in the proof of Proposition~\ref{prop:conjiso}, this provides an automorphism of the $\mathcal{O}_{K,S}[\op{C}_\ell]$-module $M$. This construction is obviously compatible with composition. 

Conversely, an $\mathcal{O}_{K,S}[\op{C}_\ell]$-automorphism of the module $M$ corresponding to $\iota(\op{C}_\ell)$ provides a change-of-basis matrix which is in the centralizer of $\iota(\op{C}_\ell)$. 
Again, this is obviously compatible with composition. 

The two constructions above are inverses, proving the claim for the centralizer. 
The claims for the normalizer are proved in the same way, changing the lower morphism in the commutative diagram from $\iota(g)$ to $\phi\circ\iota(g)$. 
\end{proof}

\begin{remark}
We will see later that the semi-linear automorphisms correspond to the action of the Galois group of $\mathcal{O}_{K,S}[\zeta_\ell]$ over $\mathcal{O}_{K,S}$ on $\op{GL}_n(\mathcal{O}_{K,S})$. Consequently, one of the contributions to the conjugacy classification is given by the Galois-orbits on the class group.
\end{remark}

\section{Modules over cyclic group rings}
\label{sec:groupring}

In this section, we provide a generalization of Reiner's classification of $\mathbb{Z}[\op{C}_\ell]$-modules, cf. \cite{reiner:1955}. Reiner's analysis of the modules over the group ring $\mathbb{Z}[\op{C}_\ell]$ is essentially based on the class group theory for cyclotomic integers. In the generalization to rings of $S$-integers,  we will therefore need some assumption on the situation, related to existence of relative integral bases.

\subsection{Relative integral bases}

As usual, denote by $\Phi_\ell(T)$ the $\ell$-th cyclotomic polynomial. If $\Phi_\ell(T)$ is not $K$-irreducible, then the degree of $\zeta_\ell$ over $K$ is a strict divisor of the degree of $\Phi_\ell(T)$. In this case, $\mathcal{O}_{K,S}[\zeta_\ell]=\mathcal{O}_{K,S}[T]/(\Psi_\ell(T))$ where $\Psi_\ell(T)$ is the minimal polynomial of $\zeta_\ell$ over $K$. 
\begin{quote}
\emph{To get a full analogue of Reiner's result, we assume that the ring $\mathcal{O}_{K,S}[T]/(\Phi_\ell(T))$ is a Dedekind domain. Some results will work under the weaker hypothesis that $\mathcal{O}_{K,S}[\zeta_\ell]$ is a Dedekind domain. We will make these cases explicit.}
\end{quote}
Note that even if $\mathcal{O}_{K,S}[\zeta_\ell]$ is a Dedekind domain, $\mathcal{O}_{K,S}[T]/(\Phi_\ell(T))$ need not be a Dedekind domain. If $\Phi_\ell(T)$ is not $K$-irreducible, then the total ring of fractions is $K[T]/(\Phi_\ell(T))$ which is a direct sum of copies of $K(\zeta_\ell)$, corresponding to the number of $K$-factors of $\Phi_\ell(T)$.

\begin{example} \label{ex:2factors}
In the case $K=\mathbb{Q}(\sqrt{-7})$ and $\ell=7$, denote by $\op{N}_7=\sum_{i=0}^6[g^i]$ the norm element in $\mathbb{Z}[\op{C}_7]$, where $g\in\op{C}_7$ is a generator. 
Then $\mathcal{O}_{K}[\op{C}_\ell]/(\op{N}_7) \cong \mathcal{O}_K[T]/(\Phi_\ell(T))$ is a fiber product of two copies of $\mathcal{O}_{K}[\zeta_7]$ over the quotient $\mathcal{O}_K[\zeta_7]/(\sqrt{-7}^3)$ where $\sqrt{-7}^3$ is the resultant of the two $K$-factors of $\Phi_7(T)=\Psi_7(T)\cdot\overline{\Psi_7(T)}$, 
\begin{eqnarray*}
 \Psi_7(T) 
 & = & T^3-(\zeta_7+\zeta_7^2+\zeta_7^4)T^2+(\zeta_7^3+\zeta_7^5+\zeta_7^6)T-1
 \\
 & = &
 T^3+\left(\frac{1-\sqrt{-7}}{2}\right)T^2-
 \left(\frac{1+\sqrt{-7}}{2}\right)T-1
 \end{eqnarray*}
 and 
\begin{eqnarray*}
 \overline{\Psi_7(T)} 
 & = & T^3-(\zeta_7^3+\zeta_7^5+\zeta_7^6)T^2+(\zeta_7+\zeta_7^2+\zeta_7^4)T-1
 \\
 & = & T^3+\left(\frac{1+\sqrt{-7}}{2}\right)T^2-
 \left(\frac{1-\sqrt{-7}}{2}\right)T-1.
\end{eqnarray*}
\end{example}

The Dedekind domain requirement is crucial because it provides a bijection between isomorphism classes of finitely generated torsion-free modules of fixed rank $n$ and the class group. The ring $\mathcal{O}_{K,S}[\zeta_\ell]$ is a Dedekind ring precisely when the relevant powers of $\zeta_\ell$ form a relative integral basis of $K(\zeta_\ell)/K$. For most of our purposes, the following statement will be sufficient.

\begin{lemma}
\label{lem:basis}
Let $K/\mathbb{Q}$ be a Galois extension with discriminant $d_K$, and let $\ell$ be a prime with $(\ell,d_K)=1$. Then 
\[
\mathcal{O}_{K(\zeta_\ell)}=\mathcal{O}_K[\zeta_\ell]\cong \mathcal{O}_K[T]/(\Phi_\ell(T)).
\]
\end{lemma}

\begin{proof}
The discriminant of $\mathbb{Q}(\zeta_\ell)/\mathbb{Q}$ is a power of $\ell$ so  that by assumption the discriminants of $K$ and $\mathbb{Q}(\zeta_\ell)$ are coprime. Then the product of the integral bases of $K/\mathbb{Q}$ and $\mathbb{Q}(\zeta_\ell)/\mathbb{Q}$ is an integral basis of $K(\zeta_\ell)/\mathbb{Q}$. In particular, any element of $\mathcal{O}_{K(\zeta_\ell)}$ is an $\mathcal{O}_K$-linear combination of $1,\zeta_\ell,\dots,\zeta_\ell^{\ell-1}$, hence these form a relative integral basis of $K(\zeta_\ell)/K$.
\end{proof}

\begin{remark}
For the case $K=\mathbb{Q}(\sqrt{-m})$ and $3\mid m$, there does not exist a relative integral basis for the extension $\mathbb{Q}(\sqrt{-m},\zeta_3)/\mathbb{Q}(\sqrt{-m})$. In particular, the ring $\mathcal{O}_{\sqrt{-m}}[\zeta_3]$ is not the maximal order of $\mathbb{Q}(\sqrt{-m},\zeta_3)$ and hence is not a Dedekind ring. This follows from \cite{bird:parry}*{Theorem I}. 
\end{remark}

\subsection{Isomorphism classification} 

We can now provide our extension of Reiner's study of isomorphism classes of modules over group rings for cyclic groups, mostly following the arguments of \cite{reiner:1955}. The situation is the following: let $K$ be a number field, let $S$ be a finite set of places containing the infinite ones, and denote by $\mathcal{O}_{K,S}$ the ring of $S$-integers in $K$. Denote by $\op{C}_\ell$ the cyclic group of order $\ell$ where $\ell$ is a prime. In some cases relevant to computations of Farrell--Tate cohomology, we will give a classification of finitely generated $\mathcal{O}_{K,S}[\op{C}_\ell]$-modules which are $\mathcal{O}_{K,S}$-free. For this, we assume that $\mathcal{O}_{K,S}[T]/(\Phi_\ell(T))$ is a Dedekind domain. Note that in this case $\mathcal{O}_{K,S}[T]/(\Phi_\ell(T))\cong\mathcal{O}_{K,S}[\zeta_\ell]$. 
Denote by  $\op{N}=\sum_{g\in\op{C}_\ell}[g]$ the norm element.
\begin{theorem}
\label{thm:fullreiner}
Let $\mathcal{O}_{K,S}$ be a ring of $S$-integers in a number field and let $\ell$ be a prime. Assume that $\mathcal{O}_{K,S}$ is a principal ideal domain and unramified at $\ell$, in particular $\mathcal{O}_{K,S}[T]/(\Phi_\ell(T))$ is a Dedekind domain. Then the isomorphism classes of finitely generated $\mathcal{O}_{K,S}$-free $\mathcal{O}_{K,S}[\op{C}_\ell]$-modules M are parametrized by 
\begin{enumerate}
\item the $\mathcal{O}_{K,S}[\zeta_\ell]$-rank $r$ of $M_{\op{N}} := \{m\in M\mid \op{N}\cdot m=0\}$,
\item the ideal class of the determinant of the $\mathcal{O}_{K,S}[\zeta_\ell]$-module $M_{\op{N}}$,
\item the $\mathcal{O}_{K,S}$-rank $s$ of $M/M_{\op{N}}$,
\item a $\op{min}(r,s)$-tuple of $\mathcal{O}_{K,S}[\zeta_\ell]^\times$-orbits of elements in $\mathcal{O}_{K,S}[\zeta_\ell]/(\zeta_\ell-1)$ (for the natural multiplication action).
\end{enumerate}
In the above, any non-negative integer 
is possible in (1) and (3). Any ideal class in (2) is possible, and any choice of tuple of orbits in (4) is possible.
\end{theorem}

\begin{proof}
Choose a generator $\gamma$ for the cyclic group $\op{C}_\ell$,  so that $\op{C}_\ell = \langle \gamma \medspace | \medspace \gamma^\ell = 1 \rangle.$
Let $M$ be an $\mathcal{O}_{K,S}$-free $\mathcal{O}_{K,S}[\op{C}_\ell]$-module. 
The set $M_{\op{N}}$ of elements of $M$ annihilated by the norm element $\op{N}$
has a natural module structure over the quotient ring $\mathcal{O}_{K,S}[\op{C}_\ell]/(\op{N})$. The kernel of the natural surjective morphism 
\[
\mathcal{O}_{K,S}[\op{C}_\ell]\to \mathcal{O}_{K,S}[T]/(\Phi_\ell(T))\colon \gamma\mapsto T
\]
is the ideal generated by $\Phi_\ell(\gamma)=\op{N}$. In particular, we get an induced isomorphism 
\[
\mathcal{O}_{K,S}[\op{C}_\ell]/(\op{N})\cong \mathcal{O}_{K,S}[T]/(\Phi_\ell(T))\cong\mathcal{O}_{K,S}[\zeta_\ell].
\]
Since $M$ is $\mathcal{O}_{K,S}$-free, the submodule $M_{\op{N}}$ embeds into a direct sum of copies of $K(\zeta_\ell)$ and hence is finitely generated and torsion-free over $\mathcal{O}_{K,S}[\zeta_\ell]$. By assumption, $\mathcal{O}_{K,S}[\zeta_\ell]$ is a Dedekind ring. In particular, finitely generated torsion-free modules over this ring are projective. The general theory of Dedekind rings and class groups implies that the projective module $M_{\op{N}}$ is of the form $\mathcal{O}_{K,S}[\zeta_\ell]^{r-1}\oplus\mathfrak{a}$ with $r$ the $\mathcal{O}_{K,S}[\zeta_\ell]$-rank of $M_{\rm N}$ and $\mathfrak{a}$ a fractional ideal of $\mathcal{O}_{K,S}[\zeta_\ell]$. The $\mathcal{O}_{K,S}[\zeta_\ell]$-module (and the restricted $\mathcal{O}_{K,S}[\op{C}_\ell]$-module) $M_{\op{N}}$ is completely determined by $r$ and the ideal class of $\mathfrak{a}$. This provides the data in (1) and (2). 

There is an inclusion of $\mathcal{O}_{K,S}[\zeta_\ell]$-modules 
$
M_{\op{N}}\supset (\gamma-1)M\supset (\zeta_\ell-1)M_{\op{N}}
$
(since ${\rm N}\cdot\gamma={\rm N}$ in $\mathcal{O}_{K,S}[{\rm C}_\ell]$,  and $\zeta_\ell=\gamma$ on $M_{\rm N}$). 
As noted above, 
\begin{center}$M_{\op{N}}\cong \mathcal{O}_{K,S}[\zeta_\ell]^{r-1}\oplus\mathfrak{a}$ (as $\mathcal{O}_{K,S}[\zeta_\ell]$-modules); \end{center}
and in this module, by standard results on Dedekind rings (as in Reiner's paper), $(\gamma-1)M$ is of the form $I_1\oplus\cdots\oplus I_{r-1}\oplus I_r\mathfrak{a}$ for ideals $(\zeta_\ell-1)\subseteq I_j\subseteq \mathcal{O}_{K,S}[\zeta_\ell]$. Hence the quotient $B:=(\gamma-1)M/(\zeta_\ell-1)M_{\op{N}}$ is a module over the quotient ring $\mathcal{O}_{K,S}[\zeta_\ell]/(\zeta_\ell-1)$, and the latter ring is isomorphic to $\mathcal{O}_{K,S}/(\ell)$ because the same is true over $\mathbb{Z}$. Consequently, the module $(\gamma-1)M$ is determined exactly by an ideal in $\left(\mathcal{O}_{K,S}[\zeta_\ell]/(\zeta_\ell-1)\right)^r\cong \left(\mathcal{O}_{K,S}/(\ell)\right)^r$.
\begin{remark}{In Reiner's work \cite{reiner:1955}, $\mathbb{Z}/(\ell)\cong\mathbb{F}_\ell$ and the ideal is simply determined by an integer $\leq r$. The same is still true whenever the prime $\ell$ is inert in the field extension $K/\mathbb{Q}$.}
\end{remark}

The quotient $M/M_{\op{N}}$ is a finitely generated torsion-free  $\mathcal{O}_{K,S}$-module: by assumption $M$ embeds into a direct sum of copies of $K$ and $K(\zeta_\ell)$ and $M_{\op{N}}$ is the part of $M$ which embeds into the $K(\zeta_\ell)$-copies. Hence $M/M_{\op{N}}$ is projective and the sequence $0\to M_{\op{N}}\to M\to M/M_{\op{N}}\to 0$ splits as $\mathcal{O}_{K,S}$-modules. The module $M$ is $\mathcal{O}_{K,S}$-free by assumption. Therefore, as $\mathcal{O}_{K,S}$-modules, $M_{\op{N}}\cong\mathcal{O}_{K,S}^a\oplus\mathfrak{b}$ and $M/M_{\op{N}}\cong\mathcal{O}_{K,S}^b\oplus\mathfrak{b}^{-1}$ for some fractional $\mathcal{O}_{K,S}$-ideal $\mathfrak{b}$. The module $M/M_{\op{N}}$ is determined (both as $\mathcal{O}_{K,S}$-module and as $\mathcal{O}_{K,S}[\op{C}_\ell]$-module) up to isomorphism by $b$ and the ideal class of $\mathfrak{b}$. Since the ideal $\mathfrak{b}$ is equivalent to the norm of the ideal $\mathfrak{a}$ in the extension $\mathcal{O}_{K,S}[\zeta_\ell]/\mathcal{O}_{K,S}$, its ideal class is determined by the one of $\mathfrak{a}$. This provides the information in (3). 

It remains to identify the $\mathcal{O}_{K,S}[\op{C}_\ell]$-module structure of $M$ in terms of the module structures of $M_{\op{N}}$ and $M/M_{\op{N}}$. The module $M$ is an extension $0\to M_{\op{N}}\to M\to M/M_{\op{N}}\to 0$ where $M_{\op{N}}$ has the $\mathcal{O}_{K,S}[\op{C}_\ell]$-structure induced from the module structure over $\mathcal{O}_{K,S}[\zeta_\ell]$ with a non-trivial $\zeta_\ell$-action, and the structure of $M/M_{\op{N}}$ is induced from $\mathcal{O}_{K,S}$, i.e., has a trivial $\zeta_\ell$-action. We noted above that $M/M_{\op{N}}$ is $\mathcal{O}_{K,S}$-projective, so that the extension splits as $\mathcal{O}_{K,S}$-module. Hence we can, as in Reiner's paper, choose an $\mathcal{O}_{K,S}$-complement $X$ of $M_{\op{N}}$ lifting $M/M_{\op{N}}$ and write $M=M_{\op{N}}\oplus X$. To write down the action of $\gamma$ in this decomposition, we note that $\gamma$ acts as $\zeta_\ell$ on $M_{\op{N}}$ and acts trivially on $M/M_{\op{N}}$. Therefore, $\gamma(x)=x+(\gamma-1)(x)$ for an $\mathcal{O}_{K,S}$-linear map $(\gamma-1)\colon M/M_{\op{N}}\to M_{\op{N}}$. From the decomposition $M=M_{\op{N}}\oplus X$,
\[
(\gamma-1)M=(\zeta_\ell-1)M_{\op{N}}\oplus(\gamma-1)X.
\]
Therefore, the map $(\gamma-1)\colon X\to M_{\op{N}}$ is determined, modulo $(\zeta_\ell-1)M_{\op{N}}$, by an $\mathcal{O}_{K,S}$-linear map $X\to B$. To describe the map $X\to  B$, we recall that $X\cong \mathcal{O}_{K,S}^b\oplus \mathfrak{b}^{-1}$ and $B:=(\gamma-1)M/(\zeta_\ell-1)M_{\rm N}$ is a submodule of $\mathcal{O}_{K,S}/(\ell)^{\oplus r}$. Then $X\to B$ factors through a map $\gamma-1\colon\mathcal{O}_{K,S}/(\ell)^{\oplus(b+1)}\to \mathcal{O}_{K,S}/(\ell)^{\oplus r}$ which surjects onto  $B$ (as quotient of $(\gamma-1)M$) by construction. The ring $\mathcal{O}_{K,S}/(\ell)$ is a finite semisimple ring, 
a product of finite extensions of $\mathbb{F}_\ell$. We can then, as in Reiner's paper, choose appropriate bases, $y_1,\dots,y_{b+1}$ of $\mathcal{O}_{K,S}/(\ell)^{\oplus(b+1)}$ and $\beta_1,\dots,\beta_r$ of $\mathcal{O}_{K,S}/(\ell)^{\oplus r}$, such that the map $(\gamma-1)$ has the form $y_j\mapsto c_j\beta_j$ for $j=1,\dots,r$ and suitable coefficients $c_j$, and $y_j\mapsto 0$ for $j>r$. This can  be done, starting from any basis, by componentwise elementary transformations over the product of fields, achieving that the representing matrix has the required shape.

It remains to figure out which $\mathcal{O}_{K,S}/(\ell)$-multiples in the above give rise to isomorphic module structures. In Reiner's paper, this is taken care of by \cite{reiner:1955}*{Lemma 4}: over $\mathcal{O}_{K,S}=\mathbb{Z}$ all the coefficients can be taken to be 1. In our more general case, the appropriate replacement of \cite{reiner:1955}*{Lemma 4} is the following statement: for a fractional $\mathcal{O}_{K,S}[\zeta_\ell]$-ideal $\mathfrak{a}$, an element $\beta\in\mathfrak{a}$ and two elements $c_1,c_2\in\mathcal{O}_{K,S}$, the two $\mathcal{O}_{K,S}[\op{C}_\ell]$-module structures on $\mathfrak{a}\oplus\mathcal{O}_{K,S}\cdot y$ given by $y\mapsto y+c_1\beta$ and $y\mapsto y+c_2\beta$, respectively, are isomorphic if and only if there exists a unit $u\in\mathcal{O}_{K,S}[\zeta_\ell]^\times$ such that $uc_1=c_2$. 
In particular, it is in general not possible to have all the coefficients be 1, but coefficients in the same orbit of the unit group give rise to isomorphic actions. The data in (5) for the above action is therefore given by the orbits of the coefficients $c_i\in \mathcal{O}_{K,S}/(\ell)\cong\mathcal{O}_{K,S}[\zeta_\ell]/(\zeta_\ell-1)$ under the multiplication action of the unit group $\mathcal{O}_{K,S}[\zeta_\ell]^\times$. 

The explicit construction of a module with the specified invariants goes through as in \cite{reiner:1955}, showing the realizability of all choices. More precisely, for a given number $r$ and ideal $\mathfrak{a}\subseteq\mathcal{O}_{K,S}[\zeta_\ell]$, we have a module $M_{\rm N}=\mathcal{O}_{K,S}[\zeta_\ell]^{r-1}\oplus\mathfrak{a}$. Then we define $M=M_{\rm N}\oplus \mathcal{O}_{K,S}^{\oplus s}$ with the $\mathcal{O}_{K,S}[{\rm C}_\ell]$-action as follows: on $M_{\rm N}$, $\gamma$ acts via $\zeta_\ell$ in the $\mathcal{O}_{K,S}[\zeta_\ell]$-module structure, and the unipotent part of the action is given by the map $\mathcal{O}_{K,S}^{\oplus s}\to M_{\rm N}$  which maps the generators $y_j$ of $\mathcal{O}_{K,S}^{\oplus s}$ to $y_j+c_j\beta_j$ for $j=1,\dots,r$ and $y_j$ for $j>r$. Here the elements $c_j$ are $\mathcal{O}_{K,S}[\zeta_\ell]$-lifts of the corresponding elements from $\mathcal{O}_{K,S}[\zeta_\ell]/(\zeta_\ell-1)$, and the elements $\beta_j$ are $M_{\rm N}$-lifts of the corresponding elements in $M_{\rm N}/(\zeta_\ell-1)M_{\rm N}$. By construction, this yields an $\mathcal{O}_{K,S}[\zeta_\ell]$-module with the specified invariants.
Namely, the $\ell$-th power of the constructed matrix preserves the filtration of $M$ by $M_{\rm N}$, and acts as identity on both $M_{\rm N}$ and $M/M_{\rm N}$ --- then it is determined uniquely by a morphism $M/M_{\rm N} \to M_{\rm N} \mod (\zeta_\ell-1)$, but that is trivial by construction (the $\ell$-th power of the matrix will already have this morphism landing in the ideal $(\ell)$ ), so the original matrix really must have order $\ell$. 
\end{proof}

\begin{remark}
It should be pointed out that the assumption that $\mathcal{O}_{K,S}[T]/(\Phi_\ell(T))$ is a Dedekind ring is important. What comes up naturally in the first step are finitely generated torsion-free modules. That classification problem is fairly complicated (even in the geometric case for torsion-free sheaves on singular curves only very few very special cases have been addressed in the literature), but over Dedekind rings reduces to the classification of finitely generated locally free modules. 
\end{remark}

We now formulate a very special case of the classification which works under the weaker assumption that $\mathcal{O}_{K,S}[\zeta_\ell]$ is a Dedekind ring but in which $\Phi_\ell(T)$ need not be $K$-irreducible. We restrict to the case where $M_{\op{N}}$ has $\mathcal{O}_{K,S}[\zeta_\ell]$-rank 1. In this case, base-change to $K$ results in one of the irreducible $K$-representations of $\op{C}_\ell$. The $\mathcal{O}_{K,S}[T]/(\Phi_\ell)$-module structure factors through a projection $\mathcal{O}_{K,S}[T]/(\Phi_\ell)\twoheadrightarrow \mathcal{O}_{K,S}[\zeta_\ell]$ and is completely determined by this. Again, the $\mathcal{O}_{K,S}[\op{C}_\ell]$-module structure of $M_{\op{N}}$ is completely determined by a fractional ideal $\mathfrak{a}$ in $\mathcal{O}_{K,S}[\zeta_\ell]$. The rest of the analysis goes through to show the following:

\begin{proposition}
\label{prop:halfreiner}
Let $\mathcal{O}_{K,S}$ be a ring of $S$-integers in a number field $K$ and let $\ell$ be a prime. Assume that $\mathcal{O}_{K,S}$ is a principal ideal domain and that $\mathcal{O}_{K,S}[\zeta_\ell]$ is a Dedekind domain. Then the isomorphism classes of finitely generated $\mathcal{O}_{K,S}$-free $\mathcal{O}_{K,S}[\op{C}_\ell]$-modules $M$ where $M_{\op{N}}$ has $\mathcal{O}_{K,S}[\zeta_\ell]$-rank 1 are parametrized by
\begin{enumerate}
\item the ideal class of the $\mathcal{O}_{K,S}[\zeta_\ell]$-module $M_{\op{N}}$,
\item the $\mathcal{O}_{K,S}$-rank of $M/M_{\op{N}}$,
\item an $\mathcal{O}_{K,S}[\zeta_\ell]^\times$-orbit of elements in $\mathcal{O}_{K,S}[\zeta_\ell]/(\zeta_\ell-1)$ (for the natural multiplication action).
\end{enumerate}
In the above, any integer $n\geq 0$ is possible in (2). Any ideal class is possible in (1), and any choice of orbit in (3) is possible.
\end{proposition}

\begin{remark}
Pulling back modules along the two projections  
\[
\mathcal{O}_{\mathbb{Q}(\sqrt{-7})}[T]/(\Phi_7(T))\to \mathcal{O}_{\mathbb{Q}(\zeta_7)}
\]
results  in  non-isomorphic modules which correspond to non-isomorphic $\mathbb{Q}(\sqrt{-7})$-representations of $\op{C}_7$. However, this effectively only amounts to different choices of generators of conjugate subgroups. If we are only interested in counting subgroups, this does not affect the end result.
\end{remark}

\section{Centralizers and normalizers}
\label{sec:centralizer}

We now need to describe centralizers and normalizers of the corresponding $\op{C}_\ell$-subgroups of $\op{GL}_n(\mathcal{O}_{K,S})$. For the purpose of this section, fix a subgroup $\iota\colon \op{C}_\ell\hookrightarrow \op{GL}_n(\mathcal{O}_{K,S})$ and the corresponding $\mathcal{O}_{K,S}[\op{C}_\ell]$-module $M$. Since our intended application is to essential rank one cases, most notably $\op{GL}_3(\mathcal{O}_{K,S})$, we assume throughout the section that the associated module $M$ is such that its associated representation over $K$ is of the form $K\times K(\zeta_\ell)$. We also assume throughout this section that the conditions of Proposition~\ref{prop:halfreiner} are satisfied.

First, we can embed $\op{GL}_n(\mathcal{O}_{K,S})\hookrightarrow \op{GL}_n(K)$. The centralizer of $\op{C}_\ell\hookrightarrow \op{GL}_n(K)$ is the automorphism group of the representation $M\otimes_{\mathcal{O}_{K,S}} K\cong K\times K(\zeta_\ell)$ of $\op{C}_\ell$ over $K$.
Under our assumption $\zeta_\ell\not\in K$, the $\op{C}_\ell$-representation $K(\zeta_\ell)$ is $K$-irreducible. In particular, 
\[
\op{Hom}_{K[\op{C}_\ell]}(K(\zeta_\ell),K)\cong \op{Hom}_{K[\op{C}_\ell]}(K,K(\zeta_\ell))\cong 0.
\]
From this, any $K[\op{C}_\ell]$-automorphism $\phi$ of $K\times K(\zeta_\ell)$ must be of the form $\phi_K\times \phi_{K(\zeta_\ell)}$ where $\phi_K$ and $\phi_{K(\zeta_\ell)}$ are $K[\op{C}_\ell]$-automorphisms of $K$ and $K(\zeta_\ell)$, respectively. Via the embedding $\op{GL}_n(\mathcal{O}_{K,S})\hookrightarrow \op{GL}_n(K)$, the same must be true for automorphisms of the $\mathcal{O}_{K,S}[\op{C}_\ell]$-modules. In terms of the centralizer as a subgroup of $\op{GL}_n(\mathcal{O}_{K,S})$, this means that the centralizer must be conjugate to a block-diagonal matrix. 
For the normalizer, similar statements apply. The only additional elements in the normalizer would come from $K$-linear automorphisms of $K(\zeta_\ell)$ which are accounted for by the Galois group $\op{Gal}(K(\zeta_\ell)/K)$. 

Now we need some induction-type theorems to determine the automorphism groups of the individual almost-direct summands of the module $M$.

\begin{lemma}
\label{lem:induction1}
Let $M$ be an $\mathcal{O}_{K,S}[\op{C}_\ell]$-module such that multiplication with the norm element $\op{N}$ is the zero map and assume that the $\mathcal{O}_{K,S}[\zeta_\ell]$-rank of $M$ is $1$. Then 
\[
\op{Aut}_{\mathcal{O}_{K,S}[\op{C}_\ell]}(M)\cong \op{Aut}_{\mathcal{O}_{K,S}[\zeta_\ell]}(M)\cong \mathcal{O}_{K,S}[\zeta_\ell]^\times. 
\]
\end{lemma}

\begin{proof}
Since the norm element $\op{N}$ annihilates $M$, it has an induced module structure for 
\[
\mathcal{O}_{K,S}[\op{C}_\ell]/(\op{N})\cong \mathcal{O}_{K,S}[\zeta_\ell].
\]
This yields a homomorphism $\op{Aut}_{\mathcal{O}_{K,S}[\op{C}_\ell]}(M)\to\op{Aut}_{\mathcal{O}_{K,S}[\zeta_\ell]}(M)$. This homomorphism is injective, since both automorphism groups embed into $\op{Aut}_{\mathcal{O}_{K,S}}(M)$. The natural restriction map along the homomorphism $\mathcal{O}_{K,S}[\op{C}_\ell]\to\mathcal{O}_{K,S}[\zeta_\ell]$ provides an inverse, establishing the first isomorphism. 

For the second isomorphism, we know that $M$ is a finitely generated projective $\mathcal{O}_{K,S}[\zeta_\ell]$-module, and our additional assumption is that its rank is $1$. Since local units can be patched to global units, the automorphism group of a finitely generated projective $\mathcal{O}_{K,S}[\zeta_\ell]$-module of rank $1$ is isomorphic to $\mathcal{O}_{K,S}[\zeta_\ell]^\times$. 
\end{proof}

\begin{lemma}
\label{lem:induction2} 
Let $M$ be an $\mathcal{O}_{K,S}[\op{C}_\ell]$-module such that multiplication with the norm element $\op{N}$ is injective and $M$ is a finitely generated projective $\mathcal{O}_{K,S}$-module of rank~$1$. Then 
\[
\op{Aut}_{\mathcal{O}_{K,S}[\op{C}_\ell]}(M)\cong \op{Aut}_{\mathcal{O}_{K,S}}(M) \cong \mathcal{O}_{K,S}^\times. 
\]
\end{lemma}

\begin{proof}
Injectivity of multiplication with the norm implies that the action of $\op{C}_\ell$ is trivial. The second isomorphism $\op{Aut}_{\mathcal{O}_{K,S}}(M)\cong \mathcal{O}_{K,S}^\times$ follows as before, by patching local units. 

An $\mathcal{O}_{K,S}[\op{C}_\ell]$-automorphism of $M$ is in particular an $\mathcal{O}_{K,S}$-automorphism, giving rise to an injective restriction map $\op{Aut}_{\mathcal{O}_{K,S}[\op{C}_\ell]}(M)\to \op{Aut}_{\mathcal{O}_{K,S}}(M)$. Since any $\mathcal{O}_{K,S}$-automorphism of $M$ automatically commutes with the trivial $\op{C}_\ell$-action we get the first isomorphism. 
\end{proof}

We can now analyse the structure of $\op{Aut}_{\mathcal{O}_{K,S}[\op{C}_\ell]}(M)$ where  $M$ is a module corresponding to a $\op{C}_\ell$-subgroup of $\op{GL}_3(\mathcal{O}_{K,S})$. Any $\mathcal{O}_{K,S}[\op{C}_\ell]$-automorphism $\phi$ of $M$ necessarily maps the submodule $M_{\op{N}}$ to itself and hence induces automorphisms $\phi_{\op{N}}$ and $\overline{\phi}$ of $M_{\op{N}}$ and $M/M_{\op{N}}$, respectively. By Lemmas~\ref{lem:induction1} and \ref{lem:induction2}, we have an induced morphism 
\[
\op{Aut}_{\mathcal{O}_{K,S}[\op{C}_\ell]}(M)\to \mathcal{O}_{K,S}[\op{C}_\ell]^\times\times \mathcal{O}_{K,S}^\times.
\]
For the split module (corresponding to the orbit of $0$ in $\mathcal{O}_{K,S}[\zeta_\ell]/(\zeta_\ell-1)$ as described in Theorem~\ref{thm:fullreiner} resp. Proposition~\ref{prop:halfreiner}), this actually describes the full centralizer. For a non-split module where there is an additional  unipotent action (corresponding to the orbit of a non-zero element $c\in \mathcal{O}_{K,S}[\zeta_\ell]/(\zeta_\ell-1)$), we have morphisms $\mathcal{O}_{K,S}[\zeta_\ell]\to \mathcal{O}_{K,S}/(\ell)$ and $\mathcal{O}_{K,S}\to \mathcal{O}_{K,S}/(\ell)$ given by reduction mod $\zeta_\ell-1$ and reduction mod $\ell$, respectively. These ring homomorphisms induce maps on the unit groups. 

\begin{lemma}
\label{lem:fiber}
Assume $M$ is the $\mathcal{O}_{K,S}[\op{C}_\ell]$-module associated to a $\op{C}_\ell$-subgroup of $\op{GL}_3(\mathcal{O}_{K,S})$ where $[K(\zeta_\ell):K]=2$. Denoting by $c\in\mathcal{O}_{K,S}[\zeta_\ell]/(\zeta_\ell-1)$ an element defining the unipotent action on $M$, we set
\[
\mathcal{O}_{K,S}[\op{C}_\ell]^\times\times^c_{\op{End}(\mathcal{O}_{K,S}/(\ell))} \mathcal{O}_{K,S}^\times=\left\{(\phi,\psi)\mid \overline{\phi} c=c\overline{\psi} \textrm{ in }\op{End}(\mathcal{O}_{K,S}/(\ell))\right\}.
\]
Then the induced morphism from the automorphism group above factors through an isomorphism
\[
\op{Aut}_{\mathcal{O}_{K,S}[\op{C}_\ell]}(M)\to \mathcal{O}_{K,S}[\op{C}_\ell]^\times\times^c_{\op{End}(\mathcal{O}_{K,S}/(\ell))} \mathcal{O}_{K,S}^\times.
\]
\end{lemma}

\begin{proof}
It remains to identify the image of the induced morphism. Let 
\[
(\phi,\psi)\in \mathcal{O}_{K,S}[\op{C}_\ell]^\times\times^c_{\op{End}(\mathcal{O}_{K,S}/(\ell))} \mathcal{O}_{K,S}^\times.
\]
To set up notation, let $M=\mathfrak{a}\oplus\op{Nm}\mathfrak{a}^{-1}$ (but $\op{Nm}\mathfrak{a}^{-1}$ is free since $\mathcal{O}_{K,S}$ is assumed to be a principal ideal domain). The action of ${\rm C}_\ell$ is specified as in Reiner's results: it sends a generator $y$ to $(\beta,y)$ where $\beta\in\mathfrak{a}$ is a choice of preimage of an element $c\in\mathcal{O}_{K,S}[\zeta_\ell]/(\zeta_\ell-1)$. Formulated differently, the action on $x\in\op{Nm}\mathfrak{a}^{-1}$ adds a specific choice of lift $\tilde{\overline{x}}\in\mathfrak{a}$ of the product  $c\overline{x}$ of the reduction of $x$ mod $\ell$ and the coefficient $c$; for notational purposes, we denote this lift $\tilde{\overline{x}}$ by $\beta(x)$. 

Now we want to determine when the action described above commutes with the automorphism $(\phi,\psi)$. If we first apply the action and then the automorphism, then we get $\phi(\beta(y))$ in the component $\mathfrak{a}$. If, on the other hand, we first apply the automorphism and then the action, we get $\beta(\psi(y))$ in the component $\mathfrak{a}$. For $\phi(\beta(y))=\beta(\psi(y))$, it is necessary and sufficient that the reductions of $\phi$ and $\psi$ to $\mathcal{O}_{K,S}/(\ell)$ satisfy $\overline{\phi}(c\cdot \overline{y})=c\cdot\overline{\psi}(\overline{y})$. This is precisely the claim.
\end{proof}

\begin{remark}
Since $\overline{\phi}$ and $\overline{\psi}$ are given by multiplication with units in $\mathcal{O}_{K,S}/(\ell)$ (the reductions of the respective units in the rings of integers), they are $\mathcal{O}_{K,S}/(\ell)$-linear. In particular, the requirement translates to $c(\overline{\phi}-\overline{\psi})$, and this is always satisfied if the reductions of $\phi$ and $\psi$ are the same. However, if $c$ is a zero-divisor, the requirement $c(\overline{\phi}-\overline{\psi})$ is strictly weaker; it is no requirement at all for $c=0$, which recovers the product description of the automorphism group for the split module. 
\end{remark}

\begin{lemma}
\label{lem:semidirect}
Assume $M$ is the $\mathcal{O}_{K,S}[\op{C}_\ell]$-module associated to a $\op{C}_\ell$-subgroup of $\op{GL}_3(\mathcal{O}_{K,S})$ with $[K(\zeta_\ell):K]=2$. In particular, $M\cong\mathfrak{a}\oplus\op{Nm}\mathfrak{a}^{-1}$ for an ideal class $\mathfrak{a}$ of $\mathcal{O}_{K,S}[\zeta_\ell]$. The group of special semilinear automorphisms of $M$ is of the form
\[
\left(\op{Aut}_{\mathcal{O}_{K,S}[\op{C}_\ell]}(M)\right)\rtimes \op{Stab}(\mathfrak{a},\op{Gal}(K(\zeta_\ell)/K)).
\]
The action is the natural Galois action on the automorphism group, viewed as fiber product of unit groups as in Lemma~\ref{lem:fiber}.
\end{lemma}

\begin{proof}
By embedding $\mathcal{O}_{K,S}[\op{C}_\ell]\hookrightarrow K[\op{C}_\ell]$, we already know that the only semilinear automorphisms that are not in the automorphism group come from the Galois-action of $\op{Gal}(K(\zeta_\ell)/K)$. However, the Galois group does not need to stabilize the isomorphism class of the module; this happens whenever we have a non-trivial Galois action on the class group of $K(\zeta_\ell)$. The semilinear automorphisms modulo the linear ones are exactly identified with the stabilizer of the ideal class $\mathfrak{a}$ in the Galois group, as claimed. 
\end{proof}

\section{Sample cases} 
\label{sec:examples}

Now we discuss a couple of sample cases to compare them to the computer calculations as sanity check. 

\subsection{Homological 3-torsion in \texorpdfstring{$\op{PGL}_3$}{PGL3} over quadratic imaginary integers}

Let $m$ be a square-free natural number with $3\nmid m$ and denote by $\mathcal{O}_{-m}=\mathcal{O}_{\mathbb{Q}(\sqrt{-m})}$. In the case where $\mathcal{O}_{-m}$ is a principal ideal domain, using Corollary~\ref{cor:conjiso} combined with Proposition~\ref{prop:halfreiner}, the conjugacy classes of embeddings $\op{C}_3\hookrightarrow\op{PGL}_3(\mathcal{O}_{-m})$ are parametrized by pairs of elements $(\mathfrak{a},c)$ where $\mathfrak{a}$ is an ideal class in $\mathcal{O}_{-m}[\zeta_3]$ and $c$ is an $\mathcal{O}_{-m}[\zeta_3]^\times$-orbit of elements in $\mathcal{O}_{-m}[\zeta_3]/(\zeta_3-1)$. Consequently, the number of conjugacy classes of subgroups is 
\[
\frac{\#\op{Cl}_{\mathbb{Q}(\sqrt{-m},\sqrt{-3})}}{\#\op{Gal}(\mathbb{Q}(\sqrt{-m},\sqrt{-3})/\mathbb{Q}(\sqrt{-m}))}\cdot\# \left((\mathcal{O}_{-m}[\zeta_3]/(\zeta_3-1))/\mathcal{O}_{-m}[\zeta_3]^\times\right)
\]
Here ${\#\op{Gal}(\mathbb{Q}(\sqrt{-m},\sqrt{-3})/\mathbb{Q}(\sqrt{-m}))} =2$, and it remains to determine the orbit set for the natural action of $\mathcal{O}_{-m}[\zeta_3]^\times$ on $\mathcal{O}_{-m}[\zeta_3]/(\zeta_3-1)$. To do this, we first note that our assumption $3\nmid m$ implies
\[
\mathcal{O}_{\mathbb{Q}(\sqrt{-m})}/(3)\cong \mathbb{F}_3[X]/(X^2+m)\cong \left\{ \begin{array}{ll}
\mathbb{F}_9 & m\equiv 1\bmod 3\\
\mathbb{F}_3\times\mathbb{F}_3 & m\equiv 2\bmod 3
\end{array}\right. 
\]
We make a case distinction, depending on the residue class of $m$ mod 3. These arguments actually do not require $\mathcal{O}_{-m}$ to be a principal ideal domain. We do them in this generality: for any $m$ coprime to $3$, the numbers of orbits below give the numbers of conjugacy classes of order-3-subgroups corresponding to the trivial ideal class. 

\begin{proposition}
Assume $m\equiv 2\bmod 3$. If the reduction morphism $\mathcal{O}_{-m}[\zeta_3]^\times\to(\mathbb{F}_3\times\mathbb{F}_3)^\times$ is surjective, then there are four $\mathcal{O}_{-m}[\zeta_3]^\times$-orbits on $\mathcal{O}_{-m}[\zeta_3]/(\zeta_3-1)$. Otherwise, the reduction morphism has image $\{(1,1),(-1,-1)\}$ and there are five orbits. 
\end{proposition}

\begin{proof}
Since $m\equiv 2\bmod 3$, we are in the case $\mathcal{O}_{\mathbb{Q}(\sqrt{-m})}/(3)\cong \mathbb{F}_3[X]/(X^2+m)$ where the units are given by $(\mathbb{F}_3\times\mathbb{F}_3)^\times\cong \mathbb{Z}/2\mathbb{Z}^{\oplus 2}$, concretely realized as the subset $\{(\pm 1,\pm 1)\}$. Note that actually $\mathbb{F}_3\times\mathbb{F}_3\cong \mathbb{F}_3[X]/(X^2+m)$, so that an integer $n\in\mathbb{Z}$ always reduces to $(\overline{n},\overline{n})$, and $\sqrt{-m}\in\mathcal{O}_{-m}$ maps to $(1,-1)$. In particular, if we consider the natural reduction morphism $\mathcal{O}_{-m}[\zeta_3]^\times\to(\mathbb{F}_3\times\mathbb{F}_3)^\times$, the elements $(1,1)$ and $(-1,-1)$ are the images of the global units $\pm 1\in \mathcal{O}_{-m}[\zeta_3]$. So there are only two possibilities for the image of the reduction map: either it is $\mathbb{Z}/2\mathbb{Z}\cong\{(1,1),(-1,-1)\}$ or it is the full group $(\mathbb{F}_3\times\mathbb{F}_3)^\times$. Consequently, there are two possibilities for the orbit set. In both of them, we have the orbits $\{(0,0)\}$, $\{(1,0),(-1,0)\}$ and $\{(0,1),(0,-1)\}$, irrespective of the image. If the image contains $(1,-1)$, then we have an orbit $\{(\pm 1,\pm 1)\}$. If the image does not contain $(1,-1)$, then we have two orbits $\{(1,1),(-1,-1)\}$ and $\{(1,-1),(-1,1)\}$. This proves the claim. 
\end{proof}

\begin{example}
We give examples that both possibilities can appear. 

Consider the case $m=2$. Using Pari/GP~\cite{PariGP}, we find that $\zeta_3-\sqrt{-2}\zeta_3+2$ is a fundamental unit of $\mathcal{O}_{-2}[\zeta_3]$. Since integers map to their reduction, $\sqrt{-2}$ maps to $(1,2)$ and $\zeta_3$ maps to $(1,1)$, the image of this unit under the reduction map is $(1,1)-(1,2)(1,1)+(2,2)=(2,1)$. Thus we are in the case where there are 4 orbits. 

Consider the case $m=5$. Using Pari/GP, we find that $-\sqrt{-5}+4(\zeta_3+1)+\zeta_3\sqrt{-5}$ is a fundamental unit of $\mathcal{O}_{-5}[\zeta_3]$. The image of this unit under the reduction map is $(2,1)+(8,8)+(1,2)=(-1,-1)$. In particular, $(1,2)$ is not in the image of the reduction map and therefore we are in the case with $5$ orbits. 
\end{example}

\begin{example}
  \label{ex:unit-index-1}
  In Table~\ref{first ten square-free}, we summarize the outcome for the first ten square-free numbers $m$ with $m\equiv 2\bmod 3$. The first column lists the integer $m$, the second column lists the class number of the composite field (relevant for the number of conjugacy classes). The third column lists the factor $$\frac{h_{\mathbb{Q}(\sqrt{-m},\sqrt{-3})}}{h_{\mathbb{Q}(\sqrt{-m})}h_{\mathbb{Q}(\sqrt{3m})}}$$ involving all the class numbers of the composite and intermediate fields, and the last column lists the number of orbits. (Equivalently, this is the information if the reduction of a fundamental unit modulo the ideal $(\zeta_3-1)$ is $\pm 1$ or not. The first case gives 5 orbits, the second 4.)
\begin{table}
\begin{center}
\begin{tabular}{|c||c||c||c|}
\hline
$m$ & $h_{\mathbb{Q}(\sqrt{-m},\sqrt{-3})}$ & \textrm{factor}& \textrm{orbit number} \\
\hline\hline
2 & 1 & 1 & 4\\\hline
5 & 2 & 1/2 & 5\\\hline
11 & 1 & 1 & 4 \\\hline
14 & 4 & 1/2 & 5\\\hline
17 & 4 & 1/2 & 5\\\hline
23 & 3 & 1 & 4\\\hline
26 & 12 & 1 & 4\\\hline
29 & 6 & 1/2 & 5\\\hline
35 & 2 & 1/2 & 5\\\hline
38 & 6 & 1/2 & 5\\\hline
\end{tabular}
\caption{Number of orbits for the first ten square-free numbers $m$ with $m\equiv 2\bmod 3$ in view of the factor $\frac{h_{\mathbb{Q}(\sqrt{-m},\sqrt{-3})}}{h_{\mathbb{Q}(\sqrt{-m})}h_{\mathbb{Q}(\sqrt{3m})}}$, cf. Example~\ref{ex:unit-index-1}.}
\label{first ten square-free}
\end{center}
\end{table}
Table~\ref{first ten square-free} suggests that there is a relationship between the number of orbits and the unit index in the class number formula, cf. \cite{kubota}. In our table, the orbit number $4$ appears precisely in those cases where $h_{\mathbb{Q}(\sqrt{-m},\sqrt{-3})}=h_{\mathbb{Q}(\sqrt{-m})}h_{\mathbb{Q}(\sqrt{3m})}$; and the orbit number $5$ appears in those cases where $h_{\mathbb{Q}(\sqrt{-m},\sqrt{-3})}=(1/2)h_{\mathbb{Q}(\sqrt{-m})}h_{\mathbb{Q}(\sqrt{3m})}$. This is actually true for all $m\leq 1000$ with $m\equiv 2 \bmod 3$. A \verb!SAGE!~\cite{Sage} script to check this is provided on this paper's webpage~\cite{github}.
\end{example}

\begin{proposition}
Assume $m\equiv 1\bmod 3$. The number of $\mathcal{O}_{-m}[\zeta_3]^\times$-orbits is given by Table~\ref{number_of_orbits}, depending on the image of the natural reduction morphism $\mathcal{O}_{-m}[\zeta_3]^\times\to(\mathbb{F}_9)^\times\cong\mathbb{Z}/8\mathbb{Z}$.
\end{proposition}
\begin{table}
\begin{center}
\begin{tabular}{|c||c|}
\hline
\textrm{image} & \textrm{orbit number} \\
\hline\hline
$\mathbb{Z}/8\mathbb{Z}$ & 2\\\hline
$\mathbb{Z}/4\mathbb{Z}$ & 3\\\hline
$\mathbb{Z}/2\mathbb{Z}$ & 5\\\hline
\end{tabular}
\caption{Number of $\mathcal{O}_{-m}[\zeta_3]^\times$-orbits depending on the image of the natural reduction morphism $\mathcal{O}_{-m}[\zeta_3]^\times\to(\mathbb{F}_9)^\times\cong\mathbb{Z}/8\mathbb{Z}$, in the case $m\equiv 1\bmod 3$.}
\label{number_of_orbits}
\end{center}
\end{table}

\begin{proof}
The subgroup $\{\pm 1\}$ is always contained in the image of the global units $\pm 1\in\mathcal{O}_{-m}[\zeta_3]$, hence the above list exhausts all possible cases. The orbits are then given by $\{0\}$ and the cosets of the image of the reduction map in $\mathbb{F}_9^\times\cong\mathbb{Z}/8\mathbb{Z}$. 
\end{proof}

\begin{example}
  \label{ex:unit-index-2}
Again, all three cases appear as can be checked by computing fundamental units using Pari/GP~\cite{PariGP}. This is most easily done in the cases where $-m\equiv 2,3\bmod 4$ since in this case $\zeta_3\sqrt{-m}$ is a primitive element giving rise to the integral basis $\{1,\zeta_3+1,\zeta_3\sqrt{-m},-\sqrt{-m}\}$. Once the fundamental unit is computed, its image under the reduction map is determined by noting that integers $n\in\mathbb{Z}$ map to the subfield $\mathbb{F}_3\subset\mathbb{F}_9$ via reduction mod 3, $\zeta_3$ maps to 1 and $\sqrt{-m}$ maps to a primitive element of the field extension $\mathbb{F}_9/\mathbb{F}_3$. If the fundamental unit maps to the reduction of an integer, the image is $\mathbb{Z}/2\mathbb{Z}$, if it maps to $\pm\sqrt{-m}$ the image is $\mathbb{Z}/4\mathbb{Z}$ and if it maps to a linear combination $\pm 1\pm\sqrt{-m}$ then it is the full image. For the first ten square-free $m$ with $m\equiv 1\bmod 3$, the results are presented in Table~\ref{1mod3}.
\begin{table}
\begin{center}
\begin{tabular}{|c||c||c||c|}
\hline
$m$ & $h_{\mathbb{Q}(\sqrt{-m},\sqrt{-3})}$ & \textrm{factor} & \textrm{orbit number} \\
\hline\hline
1 & 1 & 1 & 2\\\hline
7 & 1 & 1 & 3\\\hline
10 & 2 & 1/2 & 5\\\hline
13 & 4 & 1 & 3\\\hline
19 & 1 & 1 & 3\\\hline
22 & 2 & 1/2 & 5\\\hline
31 & 3 & 1 & 3\\\hline
34 & 4 & 1/2 & 5\\\hline
37 & 4 & 1 & 3\\\hline
43 & 1 & 1 & 3\\\hline
\end{tabular}
\end{center}
\caption{Number of orbits for the first ten square-free $m$ with $m\equiv 1\bmod 3$ in Example~\ref{ex:unit-index-2}.}
\label{1mod3}
\end{table}
Again, our table suggests that there is a relationship between the number of orbits and the unit index. It seems reasonable to expect that orbit number 2 only appears in the exceptional case $m=1$ where the unit group of $\mathcal{O}_{-m}$ is $\mu_4$; orbit number 3 should appear in those cases where $h_{\mathbb{Q}(\sqrt{-m},\sqrt{-3})}=h_{\mathbb{Q}(\sqrt{-m})}h_{\mathbb{Q}(\sqrt{3m})}$; and the orbit number $5$ should appear in those cases where $h_{\mathbb{Q}(\sqrt{-m},\sqrt{-3})}=(1/2)h_{\mathbb{Q}(\sqrt{-m})}h_{\mathbb{Q}(\sqrt{3m})}$. As in the previous example, this is true for $m\leq 1000$ with $m\equiv 1\bmod 3$, see this paper's webpage~\cite{github} for how to check this with a \verb!SAGE!~\cite{Sage} script.
\end{example}

\begin{remark}
  Currently, we have no conceptual explanation for the apparent (and experimentally supported) correlation between the factor $Q$ in the class number formula
  \[
  h_{\mathbb{Q}(\sqrt{-m},\sqrt{-3})}=\frac{Q}{2}h_{\mathbb{Q}(\sqrt{-m})}h_{\mathbb{Q}(\sqrt{3m})}
  \]
  of Kubota \cite{kubota} and the image of the reduction map $\mathcal{O}_{-m}[\zeta_3]^\times\to \mathcal{O}_{-m}/(3)^\times$. Some relation to indices of unit groups for the number rings are discussed in loc.cit., but it is not immediately apparent how to link these to the reduction map on units modulo the ramified prime.  
\end{remark}

Now it remains to determine the structure of the normalizers and compute the appropriate contributions to the Farrell--Tate cohomology. The corresponding centralizers will be of the form
\[
\mathcal{O}_{\mathbb{Q}(\sqrt{-m},\zeta_3)}^\times\times_{\op{End}(\mathcal{O}_{-m}/(3))}^c \mathcal{O}_{\mathbb{Q}(\sqrt{-m})}^\times
\]
for an $\mathcal{O}_{-m}[\zeta_3]^\times$-orbit representative $c\in\mathcal{O}_{-m}[\zeta_3]/(\zeta_3-1)$. For the case where $m\equiv 1\bmod 3$, $c$ is either $0$ or a unit, giving rise to the two possibilities 
\[
\mathcal{O}_{\mathbb{Q}(\sqrt{-m},\zeta_3)}^\times\times \mathcal{O}_{\mathbb{Q}(\sqrt{-m})}^\times\quad\textrm{ and }\quad
\mathcal{O}_{\mathbb{Q}(\sqrt{-m},\zeta_3)}^\times\times_{\left(\mathcal{O}_{\mathbb{Q}(\sqrt{-m})}/(3)\right)^\times} \mathcal{O}_{\mathbb{Q}(\sqrt{-m})}^\times,
\]
respectively. In the case $m\equiv 2\bmod 3$, there are two further possibilities for the orbits $\{(\pm 1,0)\}$ and $\{(0,\pm 1)\}$. In these cases, 
\[
\mathcal{O}_{\mathbb{Q}(\sqrt{-m},\zeta_3)}^\times\times_{\op{End}(\mathcal{O}_{-m}/(3))}^c \mathcal{O}_{\mathbb{Q}(\sqrt{-m})}^\times
\]
consists of those pairs $(\phi,\psi)\in \mathcal{O}_{\mathbb{Q}(\sqrt{-m},\zeta_3)}^\times\times \mathcal{O}_{\mathbb{Q}(\sqrt{-m})}^\times$ for which the reductions of $\phi$ and $\psi$ agree in the first or second component of $(\mathbb{F}_3\times\mathbb{F}_3)^\times\cong \mathbb{Z}/2\mathbb{Z}^{\times 2}$, respectively. The normalizers (whenever they do not already agree with the centralizers) will be extensions of the above groups by the group 
\[
\op{Gal}(\mathbb{Q}(\sqrt{-m},\zeta_3)/\mathbb{Q}(\sqrt{-m}))
\cong\mathbb{Z}/2\mathbb{Z}
\]
acting via the Galois action $\zeta_3\mapsto\zeta_3^2$ on the first factor and trivially on the second. Note that these actions are actually compatible via the reduction to $\mathcal{O}_{-m}[\zeta_3]/(\zeta_3-1)\cong\mathcal{O}_{-m}/(3)$ because the extension $\mathbb{Q}(\sqrt{-m},\zeta_3)/\mathbb{Q}(\sqrt{-m})$ is completely ramified over $(3)$.

By Dirichlet's unit theorem, 
\[
\mathcal{O}_{\mathbb{Q}(\sqrt{-m},\zeta_3)}^\times\cong \mathbb{Z}\times\mu_{3n}, \textrm{ and } \mathcal{O}_{-m}^\times\cong \mu_n
\]
where $n=2$ except in the case $m=1$ where $n=4$. The Galois action on $\mu_n$ is trivial because these are contained in $\mathcal{O}_{-m}$; and the Galois action on $\mathbb{Z}$ must also be non-trivial, i.e., given by multiplication with $-1$, since none of the non-torsion units is contained in $\mathcal{O}_{-m}$. The structure of the centralizer for the split representation is therefore
\[
\left(\mathcal{O}_{\mathbb{Q}(\sqrt{-m},\zeta_3)}^\times\times \mathcal{O}_{\mathbb{Q}(\sqrt{-m})}^\times\right)\rtimes\mathbb{Z}/2\mathbb{Z}
\cong\left(\left(\mathbb{Z}\times \mu_{3}\right)\rtimes\mathbb{Z}/2\mathbb{Z}\right)\times \mu_n^{\times 2}
\]
with $n$ as above.

For the normalizers of the non-split representations, i.e., where the orbit in $\mathcal{O}_{-m}[\zeta_3]/(\zeta_3-1)$ is different from $\{0\}$ we can again consider the two cases: 
\begin{enumerate}
\item If $m\equiv 2\bmod 3$, then $\mathcal{O}_{-m}[\zeta_3]/(\zeta_3-1)\cong\mathbb{F}_3\times\mathbb{F}_3$. The reduction map $\mu_2\cong\mathcal{O}_{-m}^\times\to (\mathbb{F}_3\times\mathbb{F}_3)^\times$ is injective. For $\mathcal{O}_{\mathbb{Q}(\sqrt{-m},\zeta_3)}^\times\cong\mathbb{Z}\times\mu_{3n}$, the reduction map $\mathcal{O}_{\mathbb{Q}(\sqrt{-m},\zeta_3)}^\times\to (\mathbb{F}_3\times\mathbb{F}_3)^\times$ is injective on $\mu_n$, the zero map on $\mu_3$ and the image of a fundamental unit depends on $m$. In the case where $c=\{(\pm 1,0)\}$ or $c=\{(0,\pm 1)\}$, the fiber product is given by the units in $\mathcal{O}_{-m}[\zeta_3]$ whose reduction has first resp. second component equal to the reduction of a unit from $\mathcal{O}_{-m}$, respectively. Since this is no condition at all, the fiber product is simply $\mathcal{O}_{-m}[\zeta_3]^\times\cong \mathbb{Z}\times\mu_{3n}$. In the case where the orbit contains $(1,-1)$, the fiber product consists of the group of units in $\mathcal{O}_{-m}[\zeta_3]$ whose reduction is of the form $(1,1)$ or $(-1,-1)$. The index of this subgroup is 1 or 2, depending on $m$, but in any case the isomorphism type of the units is again $\mathbb{Z}\times\mu_{3n}$. The Galois action here is the one we considered before, and
\[
\left(\mathcal{O}_{\mathbb{Q}(\sqrt{-m},\zeta_3)}^\times \times_{\op{End}(\mathcal{O}_{-m}/(3))}^c \mathcal{O}_{\mathbb{Q}(\sqrt{-m})}^\times\right)\rtimes\mathbb{Z}/2\mathbb{Z}
\cong\left(\left(\mathbb{Z}\times \mu_{3}\right)\rtimes\mathbb{Z}/2\mathbb{Z}\right)\times \mu_n
\]
\item If $m\equiv 1\bmod 3$, then $\mathcal{O}_{-m}[\zeta_3]/(\zeta_3-1)\cong\mathbb{F}_9$. Again the reduction map $\mu_n\cong\mathcal{O}_{-m}^\times\to \mathbb{F}_9^\times$ is injective. In any case, if $c$ is not the orbit $\{0\}$, it consists entirely of units. Hence the fiber product is the group of units of $\mathcal{O}_{-m}[\zeta_3]$ whose reduction is $\pm 1$. As an abelian group, this is again isomorphic to $\mathbb{Z}\times\mu_{3n}$ with the Galois action discussed previously, hence 
\[
\left(\mathcal{O}_{\mathbb{Q}(\sqrt{-m},\zeta_3)}^\times \times_{\op{End}(\mathcal{O}_{-m}/(3))}^c \mathcal{O}_{\mathbb{Q}(\sqrt{-m})}^\times\right)\rtimes\mathbb{Z}/2\mathbb{Z}
\cong\left(\left(\mathbb{Z}\times \mu_{3}\right)\rtimes\mathbb{Z}/2\mathbb{Z}\right)\times \mu_n
\]
\end{enumerate}

Now we have all the information we need to state the computation of the Farrell--Tate cohomology of $\op{PGL}_3(\mathcal{O}_{\mathbb{Q}(\sqrt{-m})})$.

\begin{theorem}
\label{thm:gl3_3-torsion}
Let $m\neq 3$ be a positive square-free integer and assume that $\mathcal{O}_{-m}$ is a principal ideal domain. Then
\[
\widehat{\op{H}}^\bullet\left(\op{PGL}_3(\mathcal{O}_{-m});\mathbb{F}_3\right)
\cong 
\widehat{\op{H}}^\bullet\left(\mathbb{Z}\times\mu_3;\mathbb{F}_3\right)^{\oplus u\cdot h_\lambda}\oplus 
\widehat{\op{H}}^\bullet\left(\left(\mathbb{Z}\times\mu_3\right)\rtimes
\mathbb{Z}/2\mathbb{Z};\mathbb{F}_3\right)^{\oplus u\cdot h_\mu}
\]
Here $u$ is the number of $\mathcal{O}_{-m}[\zeta_3]^\times$-orbits on $\mathcal{O}_{-m}[\zeta_3]/(\zeta_3-1)$, $h_\mu$ is the number of ideal classes in $\mathcal{O}_{-m}[\zeta_3]$ which are Galois-invariant for the natural action of $\op{Gal}(\mathbb{Q}(\sqrt{-m},\sqrt{-3})/\mathbb{Q}(\sqrt{-m}))$, and $h_\lambda$ is the number of 2-element Galois-orbits of ideal classes.
{In particular, $h_{\mathbb{Q}(\sqrt{-m},\sqrt{-3})}=2h_\lambda+h_\mu$.}
\end{theorem}

\begin{proofof}{Theorem~\ref{thm:3-torsion} and Theorem~\ref{thm:gl3_3-torsion}}
\label{proofthm11}
Both theorems follow from Brown's formula for the Farrell--Tate cohomology. By the discussion in Section~\ref{sec:prelims}, the elementary abelian $3$-subgroups of $\op{GL}_3(\mathcal{O}_{-m})$ all have rank 1. The structure of the centralizers and normalizers was described in the Dedekind ring case in Lemma~\ref{lem:fiber} and Lemma~\ref{lem:semidirect}. The same results are true in the more general case of Theorem~\ref{thm:3-torsion}, using the structure results concerning unit groups of orders, cf. \cite{neukirch}*{Chapter 1, \S12, Theorem 12.12 and Proposition 12.9}.  Theorem~\ref{thm:gl3_3-torsion} then follows from the more precise result of Proposition~\ref{prop:halfreiner} and the previous discussion of the $\mathcal{O}_{-m}[\zeta_3]^\times$-action on $\mathcal{O}_{-m}/(3)$ above.
\end{proofof}

More explicit information on the Farrell--Tate cohomology of such groups can now be obtained via the following computation included in \cite{sl2ff}: 
 
\begin{proposition} 
\label{prop:ftformula}
Let $A=\mathbb{Z}/n\mathbb{Z}\times \mathbb{Z}^r$, and let $\ell$ be an odd prime with $\ell\mid n$. Then, with $b_1,x_1,\dots,x_r$ denoting classes in degree $1$ and $a_2$ a class of degree $2$, 
\[
\widehat{\op{H}}^{\bullet}(A;\mathbb{F}_\ell)\cong \widehat{\op{H}}^\bullet(\mathbb{Z}/n\mathbb{Z};\mathbb{F}_\ell) \otimes_{\mathbb{F}_\ell}\bigwedge^\bullet \mathbb{F}_\ell^r\cong \mathbb{F}_\ell[a_2,a_2^{-1}](b_1,x_1,\dots,x_r).
\]

The Hochschild--Serre spectral sequence associated to the semi-direct product $A\rtimes\mathbb{Z}/2\mathbb{Z}$ (where $\mathbb{Z}/2\mathbb{Z}$ acts as $-1$ on $A$) degenerates and yields an isomorphism
\[
\widehat{\op{H}}^\bullet(A\rtimes\mathbb{Z}/2\mathbb{Z}; \mathbb{F}_\ell)\cong   \widehat{\op{H}}^\bullet(A;\mathbb{F}_\ell)^{\mathbb{Z}/2\mathbb{Z}}.
\]
The invariant classes are then given as linear combinations of tensor products of $a_2^{\otimes 2i}$ with even-degree exterior classes and tensor products of $a_2^{\otimes (2i+1)}$ with odd-degree exterior classes, respectively. 
\end{proposition}

The Farrell--Tate cohomology of the semidirect product groups in Proposition~\ref{prop:ftformula} is additively isomorphic to two copies of the cohomology of the dihedral group with $2n$ elements, with one copy shifted by $1$. This applies in particular to groups of the form
\[
\mathcal{O}_K[\zeta_\ell]^\times\rtimes\op{Gal}(K(\zeta_\ell)/K)\cong 
\left(\mathbb{Z}\times\mathbb{Z}/n\mathbb{Z}\right) \rtimes\mathbb{Z}/2\mathbb{Z}
\]
with $K=\mathbb{Q}(\sqrt{-m})$ 
(where the action in the semidirect product on the right is consequently given by multiplication with $-1$). 

The algebra in Theorem~\ref{thm:gl3_3-torsion} is given by the $\mathbb{Z}/2\mathbb{Z}$-invariant elements in the algebra $\mathbb{F}_3[a_2^{\pm 1}](b_1,x_1)$, where the action of $\mathbb{Z}/2\mathbb{Z}$ is by multiplication with $-1$ on all the generators. The invariant subalgebra is then generated by the classes $b_1x_1$ in degree 2, $b_1a_2$ and $x_1a_2$ in degree 3, and  $a_2^2$ in degree 4. Consequently, the Hilbert--Poincar{\'e} series for the positive degrees is 
\[
2\frac{T^2+2T^3+T^4}{1-T^4}=2\frac{T^2(1+T)^2}{1-T^4}.
\]

\subsection{Homological $3$-torsion in $\op{PGL}_3$ for real quadratic fields}
Now we discuss analogues of the above results for rings of integers in real quadratic fields. So let $m$ be a positive square-free number with $3\nmid m$, denote by $\mathcal{O}_{m}=\mathcal{O}_{\mathbb{Q}(\sqrt{m})}$ and assume that $\mathcal{O}_{m}$ is a principal ideal domain. As before, conjugacy classes are parametrized by pairs $(\mathfrak{a},c)$ with $\mathfrak{a}$ an ideal class in $\mathcal{O}_{m}[\zeta_3]$ and $c$ an $\mathcal{O}_{m}[\zeta_3]^\times$-orbit of elements in $\mathcal{O}_m[\zeta_3]/(\zeta_3-1)$. 

As in the imaginary case, $\mathcal{O}_{m}[\zeta_3]/(\zeta_3-1)\cong\mathcal{O}_m/(3)$ is 
\[
\mathbb{F}_3[X]/(X^2-m)\cong\left\{\begin{array}{ll} \mathbb{F}_3\times\mathbb{F}_3 & m\equiv 1\bmod 3\\ \mathbb{F}_9 & m\equiv 2\bmod 3\end{array}\right.
\]
An essential difference is now that the unit group $\mathcal{O}_{m}^\times\cong \mathbb{Z}\times\mu_2$ is already of rank 1. For the natural reduction morphism $\mathcal{O}_{m}^\times\to \mathcal{O}_m/(3)^\times$, various possibilities occur: for $\mathcal{O}_2$, $\mathcal{O}_5$, $\mathcal{O}_{13}$ and $\mathcal{O}_{17}$ we have maximal possible image, while for $\mathcal{O}_7$ and $\mathcal{O}_{11}$ the image is $\mathbb{Z}/2\mathbb{Z}\cong\{\pm 1\}$. The image of the natural reduction morphism $\mathcal{O}_m[\zeta_3]^\times\to \mathcal{O}_m/(3)^\times$ will then always contain the image of the reduction morphism $\mathcal{O}_{m}^\times\to \mathcal{O}_m/(3)^\times$ because $\mathcal{O}_m^\times\hookrightarrow\mathcal{O}_m[\zeta_3]^\times$ is a finite index subgroup. In particular, in those cases where already $\mathcal{O}_m^\times$ surjects onto $\mathcal{O}_m/(3)^\times$, we will have the minimal possible number of orbits. 

By Dirichlet's unit theorem, $\mathcal{O}_m^\times\cong \mathbb{Z}\times\mu_2$ and $\mathcal{O}_m[\zeta_3]^\times\cong \mathbb{Z}\times\mu_6$. The fiber product of Lemma~\ref{lem:fiber} is then (up to 2-torsion) isomorphic to $\mathbb{Z}\times\mathbb{Z}\times\mu_3$. Note that here the case distinctions in Lemma~\ref{lem:fiber} do not affect the isomorphism type of the fiber product as abelian groups (up to 2-torsion), only the index as subgroups of the product of unit groups $\mathcal{O}_{K,S}[C_\ell]^\times\times\mathcal{O}_{K,S}^\times$. The Galois action of $\op{Gal}(\mathbb{Q}(\sqrt{m},\sqrt{3})/\mathbb{Q}(\sqrt{m}))$ on $\mu_3$ is the non-trivial one mapping $\zeta_3\mapsto \zeta_3^2$. The Galois action on $\mathbb{Z}^2$ is trivial: for the summand $\mathbb{Z}\subseteq\mathcal{O}_m^\times$ this is clear, and the summand $\mathbb{Z}\subseteq\mathcal{O}_m[\zeta_3]^\times$ contains $\mathcal{O}_m^\times$ as finite-index subgroup. Consequently, the structure of the normalizer in the real quadratic case is $\mathbb{Z}^2\times(\mu_3\rtimes\mathbb{Z}/2\mathbb{Z})$. 

For the Farrell--Tate cohomology, this implies the following result:
\begin{proposition}
\label{prop:real}
Let $m$ be a positive square-free integer with $3\nmid m$ and assume that $\mathcal{O}_{m}$ is a principal ideal domain. Then 
\[
\widehat{\op{H}}^\bullet\left(\op{PGL}_3(\mathcal{O}_{m});\mathbb{F}_3\right)
\cong 
\widehat{\op{H}}^\bullet\left(\mathbb{Z}^2\times\mu_3;\mathbb{F}_3\right)^{\oplus u\cdot h_\lambda}\oplus 
\widehat{\op{H}}^\bullet\left(\mathbb{Z}^2\times\left(\mu_3\rtimes
\mathbb{Z}/2\mathbb{Z}\right);\mathbb{F}_3\right)^{\oplus u\cdot h_\mu}
\]
Here $u$ is the number of $\mathcal{O}_{m}[\zeta_3]^\times$-orbits on $\mathcal{O}_{m}[\zeta_3]/(\zeta_3-1)$, $h_\mu$ is the number of ideal classes in $\mathcal{O}_{m}[\zeta_3]$ which are Galois-invariant for the natural action of $\op{Gal}(\mathbb{Q}(\sqrt{m},\sqrt{-3})/\mathbb{Q}(\sqrt{m}))$, and $h_\lambda$ is the number of 2-element Galois orbits of ideal classes.
\end{proposition}

The Farrell--Tate cohomology algebra is $\mathbb{F}_3[a_2](b_1,x_1,y_1)$ for the case $\mathbb{Z}^2\times\mu_3$ and  $\mathbb{F}_3[a_2^{\pm 2}](b_1^3,x_1,y_1)$ for the case $\mathbb{Z}^2\times(\mu_3\rtimes\mathbb{Z}/2\mathbb{Z})$. 

\subsection{$5$-torsion in $\op{PGL}_3(\mathcal{O}_{\mathbb{Q}(\sqrt{5})})$}

We consider the Farrell--Tate cohomology of the group $\op{GL}_3(\mathcal{O}_{\mathbb{Q}(\sqrt{5})})$. We fix the embedding $\mathbb{Q}(\zeta_5)\hookrightarrow\mathbb{C}$ given by $\zeta_5\mapsto \exp(\frac{2\pi{\op{i}}}{5})$. Since $[\mathbb{Q}(\zeta_5):\mathbb{Q}]=4$, $5$-torsion in $\op{GL}_3(\mathcal{O}_{\mathbb{Q}(\sqrt{m})})$ can only appear for the quadratic subfield $K=\mathbb{Q}(\sqrt{5})$. The minimal polynomial of $\zeta_5$ factors as
\[
\Phi_4(X)=\left(X^2-\frac{\sqrt{5}-1}{2}X+1\right)\left(X^2+\frac{\sqrt{5}+1}{2}X+1\right)
\]
Since $\zeta_5+\zeta_5^{-1}\in\mathcal{O}_{\mathbb{Q}(\sqrt{5})}$, $\zeta_5^2=(\zeta_5+\zeta_5^{-1})\zeta_5-1$ is an $\mathcal{O}_{\mathbb{Q}(\sqrt{5})}$-linear combination of $1$ and $\zeta_5$, hence $\{1,\zeta_5\}$ is a relative integral basis. We can employ Proposition~\ref{prop:halfreiner} to obtain the conjugacy classification of $\op{C}_5$-subgroups in $\op{GL}_3(\mathcal{O}_{\mathbb{Q}(\sqrt{5})})$ and to describe the Farrell--Tate cohomology. 

The class group of $\mathcal{O}_{\mathbb{Q}(\zeta_5)}$ is trivial, and for an $\mathcal{O}_{\mathbb{Q}(\sqrt{5})}[\op{C}_5]$-module $M$ whose $\mathcal{O}_{\mathbb{Q}(\sqrt{5})}$-rank is 3 the rank of $M/M_{\op{N}}$ has to be 1. Consequently, the number of conjugacy classes equals the number of $\mathcal{O}_{\mathbb{Q}(\zeta_5)}^\times$-orbits of elements in $\mathcal{O}_{\mathbb{Q}(\zeta_5)}/(\zeta_5-1)\cong\mathbb{F}_5$. To determine the image of $\mathcal{O}_{\mathbb{Q}(\zeta_5)}^\times\to \mathbb{F}_5^\times$ we consider the composition with the natural inclusion $\mathcal{O}_{\mathbb{Q}(\sqrt{5})}^\times\hookrightarrow \mathcal{O}_{\mathbb{Q}(\zeta_5)}^\times$. Now we note that the fundamental unit of $\mathbb{Q}(\sqrt{5})$ is given by $\frac{\sqrt{5}+1}{2}$, whose image under the reduction map 
\[
\mathcal{O}_{\mathbb{Q}(\sqrt{5})}\to \mathcal{O}_{\mathbb{Q}(\sqrt{5})}/(\sqrt{5})\cong \mathbb{F}_5
\]
is $3\in\mathbb{F}_5^\times$. Consequently, the map $\mathcal{O}_{\mathbb{Q}(\zeta_5)}^\times\to \mathbb{F}_5^\times$ is surjective and there are 2 orbits. 

The structure of the normalizers and the structure of the Farrell--Tate cohomology can be determined as in the previous case of homological 3-torsion over real quadratic number rings. Consequently, the mod 5 Farrell--Tate cohomology of $\op{GL}_3(\mathcal{O}_5)$ is of the form 
\[
\widehat{\op{H}}^\bullet\left(\op{PGL}_3(\mathcal{O}_{5});\mathbb{F}_5\right)
\cong 
\widehat{\op{H}}^\bullet\left(\mathbb{Z}^2\times\left(\mu_5\rtimes
\mathbb{Z}/2\mathbb{Z}\right);\mathbb{F}_5\right)^{\oplus 2}
\]

\subsection{$7$-torsion in $\op{PGL}_3(\mathcal{O}_{\mathbb{Q}(\sqrt{-7})})$}
 
 We consider the Farrell--Tate cohomology of the  group $\op{GL}_3(\mathcal{O}_{\mathbb{Q}(\sqrt{-7})})$ with $\mathbb{F}_7$-coefficients. 
 We fix the embedding $\mathbb{Q}(\zeta_7) \hookrightarrow \mathbb{C}$, 
 $\zeta_7 \mapsto \exp(\frac{2\pi{\op{i}}}{7})$.
 For the minimal polynomial of $\zeta_7$ over $\mathbb{Q}(\sqrt{-7})$, we choose the first of the two factors from Example \ref{ex:2factors}.
 From that polynomial, we see that $\zeta_7^3$ is a $\mathcal{O}_{\mathbb{Q}(\sqrt{-7})}$-linear combination of $1,\zeta_7$ and $\zeta_7^2$. This implies that $1,\zeta_7,\zeta_7^2$ is a relative integral basis of $\mathbb{Z}[\zeta_7]$ over $\mathcal{O}_{\mathbb{Q}(\sqrt{-7})}$. However, in this case, the ring $\mathcal{O}_{\mathbb{Q}(\sqrt{-7})}[T]/(\Phi_7(T))$ is not a Dedekind domain because the cyclotomic polynomial $\Phi_7(T)$ decomposes as product of two polynomials of degree $3$ over $\mathbb{Q}(\sqrt{-7})$. We have to apply the partial result Proposition~\ref{prop:halfreiner} to get a classification. This is enough for our purposes as we are interested in those modules which are $\mathcal{O}_{\mathbb{Q}(\sqrt{-7})}$-free of rank $3$.
 
 The class group $\op{Cl}(\mathcal{O}_{\mathbb{Q}(\sqrt{-7})})$ is trivial, every ideal class of $\mathbb{Q}(\zeta_7)$ has a basis as $\mathcal{O}_{\mathbb{Q}(\sqrt{-7})}$-module. We can therefore apply Reiner's results to determine the conjugacy classes of $\op{C}_7$-subgroups in $\op{PGL}_3(\mathcal{O}_{\mathbb{Q}(\sqrt{-7})})$. Since the class group of $\op{Cl}(\mathcal{O}_{\mathbb{Q}(\zeta_7)})$ is also trivial, there is a unique conjugacy class of cyclic subgroups of order $7$, corresponding to the free rank one $\mathbb{Z}[\zeta_7]$-module. 
 
 The centralizer in $\op{PGL}_3(\mathcal{O}_{\mathbb{Q}(\sqrt{-7})})$ of the cyclic group of order $7$ is the usual unit group
 \[
 \mathcal{O}_{\mathbb{Q}(\zeta_7)}^\times\cong\mathbb{Z}^2\times \mu_{14} 
 \]
 and the normalizer is an extension of this unit group by the Galois group of the extension $\mathbb{Q}(\zeta_7)/\mathbb{Q}(\sqrt{-7})$ acting in the obvious way. Therefore, the normalizer of the cyclic group is of the form
 \[
 \mathcal{O}_{\mathbb{Q}(\zeta_7)}^\times\rtimes \op{Gal}(\mathbb{Q}(\zeta_7)/\mathbb{Q}(\sqrt{-7}))\cong 
  \left(\mathbb{Z}^2\times \mu_{14}\right)\rtimes\mathbb{Z}/3\mathbb{Z}. 
 \]
 The Galois action on the group $\mu_{14}$ is induced from the natural embedding $\mathbb{Z}/3\mathbb{Z}\hookrightarrow \op{Aut}(\mu_{14})$, and the action on $\mathbb{Z}^2$ is given by the $2$-dimensional rotation representation. 
 
 \begin{proposition}
 \[
 \widehat{\op{H}}^\bullet(\op{PGL}_3(\mathcal{O}_{\mathbb{Q}(\sqrt{-7})});
 \mathbb{F}_7)\cong
 \widehat{\op{H}}^\bullet(\mathcal{O}_{\mathbb{Q}(\zeta_7)}^\times;\mathbb{F}_7)^{\mathbb{Z}/3\mathbb{Z}}
 \cong\mathbb{F}_7[a_2^{\pm 3}](b_1^5,x_1\wedge y_1).
 \]
 \end{proposition}
 
 \begin{proof}
 The first isomorphism follows since the Hochschild--Serre spectral sequence degenerates, essentially because the $\mathbb{F}_7$-cohomology of $\mathbb{Z}/3\mathbb{Z}$ is trivial. Now we determine the action and invariant subalgebra of 
 \begin{eqnarray*}
 \widehat{\op{H}}^\bullet(\mathcal{O}_{\mathbb{Q}(\zeta_7)}^\times;\mathbb{F}_7) & 
 \cong & \widehat{\op{H}}^\bullet(\mathbb{Z}^2\times\mathbb{Z}/14\mathbb{Z};
 \mathbb{F}_7) \\
 & \cong & \mathbb{F}_7[a_2](b_1,x_1,y_1)
 \end{eqnarray*}
 The action on the degree $2$ of $\op{H}^\bullet(\mathbb{Z}/7\mathbb{Z};\mathbb{Z})\cong \mathbb{F}_7[a_2]$ is via the dual of the natural embedding $\mathbb{Z}/3\mathbb{Z}\hookrightarrow \op{Aut}(\mu_{14})$. The induced action in degree $2n$ of $\mathbb{F}_7[a_2^n]$ is the $n$-th tensor power. Therefore, the invariant subring is $\mathbb{F}_7[a_2^3]$. In the cohomology ring $\op{H}^\bullet(\mathbb{Z}/7\mathbb{Z};\mathbb{F}_7)\cong \mathbb{F}_7[a_2](b_1)$, the representation in degree $2n-1$ is the same as in degree $2n$, hence the invariant subring here is $\mathbb{F}_7[a_2^3](b_1^5)$. 
 
 The action of $\mathbb{Z}/3\mathbb{Z}$ on $\mathbb{Z}^2$ is the rotation representation. Therefore, on the cohomology $\op{H}^\bullet(\mathbb{Z}^2;\mathbb{F}_7)\cong \mathbb{F}_7(x_1,y_1)$ we have the dual of the rotation representation in degree 1. The rotation is a permutation representation by $x_1\mapsto y_1, y_1\mapsto -x_1-y_1$. From this, we see that $\bigwedge^2\mathbb{Z}^2$ is the trivial $\mathbb{Z}/3\mathbb{Z}$-representation. Hence the invariant subalgebra is the exterior algebra generated by $x_1\wedge y_1$. The result now follows. 
 \end{proof}
 
 The Hilbert--Poincar{\'e} series for the invariant algebra is 
 \[
 \frac{(1+T^5)(1+T^2)}{1-T^6}.
 \]

\subsection{$\ell$-torsion in $\op{PGL}_\ell(\mathbb{Z})$}

Another example that can be handled along the lines of the $\op{PGL}_3(\mathcal{O}_{\mathbb{Q}(\sqrt{-m})})$-samples is the $\ell$-torsion in the case $\op{PGL}_\ell(\mathbb{Z})$. By Reiner's result \cite{reiner:1955}, there are $2h_{\mathbb{Q}(\zeta_\ell)}$ conjugacy classes of $\op{C}_\ell$-subgroups in $\op{PGL}_\ell(\mathbb{Z})$. By Dirichlet's unit theorem, $\mathcal{O}_{\mathbb{Q}(\zeta_\ell)}^\times\cong \mathbb{Z}^{\frac{\ell-3}{2}}\times\mu_{2\ell}$. Therefore, the normalizers for the two conjugacy classes belonging to a given ideal class are of the form $(\mathbb{Z}^{\frac{\ell-3}{2}}\times\mu_{2\ell})\rtimes\mathbb{Z}/(\ell-1)\mathbb{Z}$ and $(\mathbb{Z}^{\frac{\ell-3}{2}}\times_{\mathbb{Z}/(\ell-1)\mathbb{Z}}\mu_{2\ell})\rtimes\mathbb{Z}/(\ell-1)\mathbb{Z}$, respectively. For the special case $\ell=5$, the Farrell--Tate cohomology of $\op{PGL}_5(\mathbb{Z})$ is then twice that for $\op{PSL}_4(\mathbb{Z})$ (where the latter was computed in \cite{psl4z}). 

\section{Machine computations}
\label{sec:machine_computations}
We ran a machine computation, starting with Vorono\"i cell complexes with an action of $\op{GL}_3(\mathcal{O}_{\mathbb{Q}(\sqrt{-m})})$,
constructed by Sch\"onnenbeck's software~\cite{Sebastian}.
More precisely,  $\op{GL}_3(\mathcal{O}_{\mathbb{Q}(\sqrt{-m})})$ acts on a certain space of quadratic forms,
and there is an equivariant retraction to Ash's well-rounded retract \cite{Ash}.
On the latter co-compact space, a suitable form of Vorono\"i's algorithm yields an explicit cell structure with cell stabilizers and computable quotient space,
as described by Braun, Coulangeon, Nebe and Sch\"onnenbeck \cite{BCNS}.

To determine the parameters $\lambda$ and $\mu$ of Theorem \ref{thm:3-torsion}, 
we want to extract the torsion subcomplexes from these Vorono\"i cell complexes.
\begin{definition}
For $\ell$ a prime number, the $\ell$-torsion subcomplex is the set of all cells with stabilizers containing some element(s) of order $\ell$.
\end{definition}
For the $\ell$-torsion subcomplex to be guaranteed to be a cell complex, 
and to consist only of fixed points of order-$\ell$-elements (so to coincide with the $\ell$-singular part),
we need a rigidity property: 
We want each cell stabilizer to fix its cell pointwise.
This rigidity property is lacking on our Vorono\"i cell complexes.
In theory, it is always possible to obtain this rigidity property via the barycentric subdivision. 
 However, the barycentric subdivision of an $n$-dimensional cell complex can multiply the number of cells by $(n+1)!$ and thus easily let the memory stack overflow. 
In previous work of the authors~\cite{psl4z}, a cell subdivision algorithm was introduced (``rigid facets subdivision''), 
which refines the cell structure to get the rigidity property, in an efficient enough way to treat the $\op{GL}_3(\mathcal{O}_{\mathbb{Q}(\sqrt{-m})})$-cases.
This allows us to extract the $\ell$-torsion subcomplexes.

We applied the algorithm to 
\begin{itemize}
 \item  the  $\op{PGL}_3(\mathbb{Z}[i])$ cell complex of Dutour Sikiric~\cite{dutour:ellis:schuermann} and 
 \item the $\op{GL}_3(\mathcal{O}_{-m})$ cell complexes for  $m\in \{1,2,7,11,15,19,5
 \}$ of Sch\"onnen\-beck~\cite{Sebastian};
\end{itemize}
then extracted the $3$-torsion subcomplex, and finally reduced it using their pertinent methods~\cite{accessingFarrell}. The software and cell complex data for these computations is available publicly~\cite{github}.

\begin{outcome}
For  $\op{GL}_3(\mathcal{O}_{-m})$,  $m\in \{1,2,7,11,15,19,5\}$, the $3$-torsion subcomplex can be reduced to a graph.
The quotient of the reduced subcomplex by the $\op{GL}_3(\mathcal{O}_{-m})$-action consists of 
$\lambda$ connected components of type~$\circlegraph$ (an edge with its endpoints identified), 
and $\mu$ connected components of type~$\edgegraph$ (an edge without identifications on its endpoints).
The counts of $\lambda$ and $\mu$ are given in Table~\ref{results-table}. 
\end{outcome}
We were not able to extend Table~\ref{results-table} with the resources available to us, since the cases treated there are the ones of lowest complexity among imaginary quadratic fields, and already they took 2 to 6 months of processing at more than half of the 64 GB, resp. 100 GB, rapid access memory of the two workstations of Gabor Wiese's research group (around 3GHz, resp. 2GHz, on Linux Ubuntu, the second workstation with \textsc{Magma}~\cite{magma} for Sch\"onnenbeck's construction of the cell complex, both workstations with \textsc{Gap}~\cite{GAP4} for the rigid facets subdivision algorithm). We ran the $m=6$ case for 12 months, but it did not complete.
\begin{corollary}
Table~\ref{results-table} specifies, within its scope, the values of the parameters $\lambda$ and $\mu$ in Theorem~\ref{thm:3-torsion}.
\end{corollary}
\begin{proof}
On each connected component of type~$\circlegraph$, the vertex stabilizer equals the edge stabilizer, 
and is of isomorphism type~$\op{C}_3$, denoting a cyclic group with $3$ elements.
On each connected component of type~$\edgegraph$, the edge stabilizer is of isomorphism type~$\op{C}_3$, 
and the two vertex stabilizers are of isomorphism type~$\op{D}_3$, denoting a dihedral group with $3\cdot 2$ elements,
but not conjugate to each other in $\op{GL}_3(\mathcal{O}_{-m})$.
From previous work of the authors~\cites{accessingFarrell,psl4z}, 
we know that since each connected component has been reduced to one edge, 
each conjugacy class of order-$3$-subgroups in $\op{GL}_3(\mathcal{O}_{-m})$
is represented precisely by one of these edges' stabilizers.
So we have $\lambda$, respectively $\mu$, conjugacy classes of order-$3$-subgroups which are not contained, respectively contained, in a copy of $\op{D}_3$ in $\op{GL}_3(\mathcal{O}_{-m})$.
\end{proof}

The reduced $\ell$-torsion subcomplex furthermore allows us to compute the mod $\ell$ Farrell-Tate cohomology of $\op{GL}_3(\mathcal{O}_{-m})$
via the equivariant spectral sequence (a.k.a. the isotropy spectral sequence).
Because there is up to conjugating isomorphism just one inclusion $\op{C}_3 \to \op{D}_3$, 
the $d_1^{p,q}$-differentials of the equivariant spectral sequence with 
$\mathbb{F}_3$-coefficients on the $3$-torsion subcomplex have the maximal possible ranks, 
i.e., they are surjective whenever both domain of definition and codomain contain $3$-torsion. 
The subsequent computation of the 
$E_2=E_\infty$-page of the isotropy spectral sequence agrees with the results formulated in Theorem~\ref{thm:gl3_3-torsion}.

\appendix
\section{Explicit order 3 matrices over \texorpdfstring{$\mathbb{Z}[\sqrt{-5}]$}{O-5}}
\label{sec:o5}

We briefly discuss how arguments as in the proof of  Theorem~\ref{thm:fullreiner} can actually be used to produce explicit matrices of finite order. The sample case we discuss is $\op{GL}_3(\mathcal{O}_{\mathbb{Q}(\sqrt{-5})})$, i.e., invertible $3\times 3$-matrices of order 3 with entries in the number ring $\mathcal{O}_{\mathbb{Q}(\sqrt{-5})}=\mathbb{Z}[\sqrt{-5}]$, which we denote by $\mathcal{O}_{-5}$ in the following. Note that our proof of Theorem~\ref{thm:fullreiner} does not apply to that case. Nevertheless, the first part of the proof allows us to provide matrices of finite order, but the conjugacy discussion in the end of the proof requires choices of suitable bases which does not work for non-free projective modules. In any case, we hope that the explicit arguments below clarify a bit the ideas behind the proof of Theorem~\ref{thm:fullreiner} as well as the issues in the general non-PID case.

We first describe the $\mathcal{O}_{-5}$-structure of the relevant modules $M$ over $\mathcal{O}_{-5}[\op{C}_3]$. By the arguments in the proof of Theorem~\ref{thm:fullreiner}, the submodule $M_{\op{N}}\subseteq M$ annihilated by the norm element has a structure of module over $\mathcal{O}_{-5}[\zeta_3]$. This is a Dedekind ring, cf. Lemma~\ref{lem:basis}, and we have a relative integral basis $\{1,\zeta_3\}$ for $\mathcal{O}_{-5}[\zeta_3]$ over $\mathcal{O}_{-5}$. A slight change of basis allows us to write
\[
\mathcal{O}_{-5}[\zeta_3]\cong \mathcal{O}_{-5}\cdot\zeta_3\oplus \mathcal{O}_{-5}\cdot(\zeta_3-1)
\]
as $\mathcal{O}_{-5}$-modules. The ideal class group of $\mathcal{O}_{-5}[\zeta_3]$ is isomorphic to $\mathbb{Z}/2\mathbb{Z}$, with a representative of the nontrivial ideal class given by $(3,\sqrt{-5}\zeta_3-1)$. Under the above isomorphism $\mathcal{O}_{-5}[\zeta_3]\cong \mathcal{O}_{-5}\cdot\zeta_3\oplus \mathcal{O}_{-5}\cdot(\zeta_3-1)$, the ideal $(3,\sqrt{-5}\zeta_3-1)$ maps to $(3,\sqrt{-5}-1)\cdot\zeta_3$ in the first factor (given by setting $\zeta_3-1=0$) and to $(3,-1)$ in the second factor (given by setting $\zeta_3=0$). In particular, as $\mathcal{O}_{-5}$-module, we have a decomposition 
\[
(3,\sqrt{-5}\zeta_3-1)\cong (3,\sqrt{-5}-1)\cdot\zeta_3\oplus \mathcal{O}_{-5}\cdot (\zeta_3-1). 
\]
In particular, this implies that the relative norm map
\[
\mathbb{Z}/2\mathbb{Z}\cong \op{Cl}(\mathcal{O}_{-5}[\zeta_3])\to \op{Cl}(\mathcal{O}_{-5})\cong\mathbb{Z}/2\mathbb{Z}
\]
is the nontrivial map. The $\mathcal{O}_{-5}[\op{C}_3]$-modules $M$ which are free of rank 3 as $\mathcal{O}_{-5}$-modules have a direct sum decomposition $M\cong M_{\op{N}}\oplus M/M_{\op{N}}$ (as $\mathcal{O}_{-5}$-modules), and there are two possibilities: the first has $M_{\op{N}}\cong \mathcal{O}_{-5}[\zeta_3]\cong \mathcal{O}_{-5}^{\oplus 2}$ and $M/M_{\op{N}}\cong \mathcal{O}_{-5}$, and the second has $M_{\op{N}}\cong (3,\sqrt{-5}\zeta_3-1)\cong \mathcal{O}_{-5}\oplus(3,\sqrt{-5}-1)$ and $M/M_{\op{N}}\cong (2,\sqrt{-5}+1)$ (where we have chosen a different ideal representative $(2,\sqrt{-5}+1)$ of the non-trivial element of the ideal class group of $\mathcal{O}_{-5}$). 

The $\mathcal{O}_{-5}[\op{C}_3]$-module structures are now determined by giving the action of the element $[\gamma]\in \mathcal{O}_{-5}[\op{C}_3]$ corresponding to a chosen generator $\gamma$ of $\op{C}_3$. By the description of the action in Theorem~\ref{thm:fullreiner}, the element $[\gamma]$ acts on $M_{\op{N}}$ via multiplication by $\zeta_3$. On $M/M_{\op{N}}$, the action is given by adding the image under an $\mathcal{O}_{-5}$-linear map $M/M_{\op{N}}\to M_{\op{N}}$.

In the first case, we get the following module, as well as explicit representing matrices for the order 3 elements of $\op{GL}_3(\mathcal{O}_{-5})$. We have $M/M_{\op{N}}\cong\mathcal{O}_{-5}$ and $M_{\op{N}}\cong \mathcal{O}_{-5}[\zeta_3]\cong\mathcal{O}_{-5}^{\oplus 2}$. As basis of $M_{\op{N}}\oplus M/M_{\op{N}}$ we choose $1\in M/M_{\op{N}}$ and $\{1,\zeta_3\}$ in $M_{\op{N}}$. The element $[\gamma]\in \mathcal{O}_{-5}[\op{C}_3]$ acts on $M_{\op{N}}$ by multiplication with $\zeta_3$, and on $M/M_{\op{N}}$ by $1\mapsto 1+\beta$ for $\beta\in M_{\op{N}}$. 
In this basis, the representing matrix for $[\gamma]$ is 
\footnotesize
\[
\left(\begin{array}{ccc} 0&-1&a\\1&-1&b\\0&0&1\end{array}\right).
\]
\normalsize
The $2\times 2$-block in the upper left is the representing matrix for multiplication by $\zeta_3$ on $\mathcal{O}_{-5}[\zeta_3]$ with basis $\{1,\zeta_3\}$, and $a,b\in\mathcal{O}_{-5}$ determine the element $a+b\zeta_3\in  M_{\op{N}}\cong\mathcal{O}_{-5}[\zeta_3]$ which is the image of $1\in M/M_{\op{N}}\cong\mathcal{O}_{-5}$ under the action of $[\gamma]$.

The second case is slightly more complicated. We have $M/M_{\op{N}}\cong (2,\sqrt{-5}+1)$, henceforth denoted by $\mathfrak{a}$, and $M_{\op{N}}\cong  (3,\sqrt{-5}-1)\oplus \mathcal{O}_{-5}$, the first summand henceforth denoted by $\mathfrak{b}$. Since the ideals $\mathfrak{a}$ and $\mathfrak{b}$ are non-principal, it is impossible to choose $\mathcal{O}_{-5}$-bases for $M_{\op{N}}$ or $M/M_{\op{N}}$ individually. However, we may choose a basis for $M_{\op{N}}\oplus M/M_{\op{N}}$ as follows: since the ideals $\mathfrak{a}$ and $\mathfrak{b}$ are relatively prime, we have the natural surjection
\[
f\colon \mathfrak{a}\oplus\mathfrak{b}\to \mathfrak{a}+\mathfrak{b}=\mathcal{O}_{-5}\colon (x,y)\mapsto x-y.
\]
The kernel of this map $f$ is $\mathfrak{a}\mathfrak{b}=(-\sqrt{-5}+1)\mathcal{O}_{-5}$. In particular, an $\mathcal{O}_{-5}$-basis for $\mathfrak{a}\oplus\mathfrak{b}$ is given by $(-2,-3)$ (obtained by lifting $1$ along $f$) and $(-\sqrt{-5}+1,-\sqrt{-5}+1)$ (obtained as image of the generator of $\ker f$ under the inclusion into $\mathfrak{a}\oplus\mathfrak{b}$). We can now write down a choice of basis for 
\[
M_{\op{N}}\oplus M/M_{\op{N}}\cong \mathfrak{b}\cdot\zeta_3\oplus \mathcal{O}_{-5}\cdot(\zeta_3-1)\oplus \mathfrak{a}
\]
given by $(-3\cdot\zeta_3,0,-2)$, $(0,1,0)$ and $((-\sqrt{-5}+1)\cdot\zeta_3,0,-\sqrt{-5}+1)$. Note that the decomposition $M_{\op{N}}\cong \mathfrak{b}\cdot\zeta_3\oplus\mathcal{O}_{-5}\cdot(\zeta_3-1)$ is the one induced from the decomposition $\mathcal{O}_{-5}[\zeta_3]\cong\mathcal{O}_{-5}\cdot\zeta_3\oplus \mathcal{O}_{-5}\cdot(\zeta_3-1)$ corresponding to the relative integral basis $(\zeta_3,\zeta_3-1)$. 

The action of $[\gamma]\in\mathcal{O}_{-5}[\zeta_3]$ is now given as follows: on $M_{\op{N}}$, $[\gamma]$ acts by multiplication  with $\zeta_3$; and on $M/M_{\op{N}}\cong\mathfrak{a}$, $[\gamma]$ acts  by sending an element $x\in \mathfrak{a}$ to $x+\lambda(x)$ where $\lambda\colon \mathfrak{a}\to M_{\op{N}}\cong \mathfrak{b}\cdot\zeta_3\oplus \mathcal{O}_{-5}\cdot(\zeta_3-1)$ is an $\mathcal{O}_{-5}$-linear map. To describe such a linear map, we note that we have an isomorphism
\[
\phi\colon\mathfrak{a}\to\mathfrak{b}\colon 2\mapsto -(\sqrt{-5}-1), \sqrt{-5}+1\mapsto 3
\]
The fact that this is a well-defined map (as well as the well-definedness of the inverse) follows from the well-known equality $2\cdot 3+(\sqrt{-5}+1)(\sqrt{-5}-1)=0$. In particular, the composition $\mathfrak{a}\xrightarrow{\lambda} M_{\op{N}}\xrightarrow{\op{pr}_1} \mathfrak{b}$ is an $\mathcal{O}_{-5}$-multiple of this isomorphism (because $\op{Hom}_{\mathcal{O}_{-5}}(\mathfrak{a},\mathfrak{b})\cong \mathfrak{a}^\vee\otimes\mathfrak{b}\cong \mathcal{O}_{-5}$). The second component of a linear map $\lambda\colon \mathfrak{a}\to M_{\op{N}}$ is a linear form $\mathfrak{a}\to \mathcal{O}_{-5}$ which can be described as follows: from before, we know $\mathfrak{a}\otimes\mathfrak{b}=\mathfrak{a}\mathfrak{b}=(-\sqrt{-5}+1)\mathcal{O}_{-5}$, hence a linear map $\mathfrak{a}\to\mathcal{O}_{-5}$ always has the form
\[
\psi\colon \mathfrak{a}\to\mathcal{O}_{-5}\colon x\mapsto \frac{x\cdot y}{(-\sqrt{-5}+1)}
\]
for $y\in \mathfrak{b}$. More precisely, for an element $3a+(\sqrt{-5}+1)b$ with $a,b\in\mathcal{O}_{-5}$, we have $\psi(2)=(-1-\sqrt{-5})a-2b$ and $\psi(\sqrt{-5}+1)=(2-\sqrt{-5})a+(1+\sqrt{-5})b$.

We are now ready to combine everything to obtain matrices representing multiplication with $[\gamma]$ in the respective $\mathcal{O}_{-5}[\op{C}_3]$-module structures. We use the basis $(-3\cdot\zeta_3,0,-2)$, $(0,1,0)$ and $((-\sqrt{-5}+1)\cdot\zeta_3,0,-\sqrt{-5}+1)$ given above. Multiplication by $[\gamma]$ is given by multiplication with $\zeta_3$ on the first two summands and by addition of $\lambda(x)$ on the third summand. This leads to the following
\[\left(
\begin{array}{ccc}
(1-\sqrt{-5})x-8 & 3 & (\sqrt{-5}+2)x-3\sqrt{-5}+3 \\
-3-(1+\sqrt{-5})a+2b & 1 & -\sqrt{-5}+1+3a+(\sqrt{-5}-1)b\\
2x-3\sqrt{-5}-3 & \sqrt{-5}+1 & (\sqrt{-5}-1)x+7
\end{array}
\right)\]
where $x,a,b\in\mathcal{O}_{-5}$ give rise to the linear map 
\[
\lambda\colon \mathfrak{a}\to\mathfrak{b}\oplus\mathcal{O}_{-5}\colon z\mapsto \left(x\cdot\phi(z),\frac{z(3a+(\sqrt{-5}+1)b)}{(-\sqrt{-5}+1)}\right).
\]

According to Theorem~\ref{thm:fullreiner} resp. Proposition~\ref{prop:halfreiner}, the matrices written above provide representatives for all conjugacy classes of $\op{C}_3$-subgroups in $\op{GL}_3(\mathcal{O}_{-5})$. 

It remains to discuss conjugacy relations among the infinitely many matrices written out above. The first case corresponding to the principal ideal class is easier: 
conjugation by $\op{e}_{13}(b)$ and $\op{e}_{23}(x)$ allows us to reduce to matrices of the form 
\footnotesize
\[
\left(\begin{array}{ccc} 0&-1&a\\1&-1&0\\0&0&1\end{array}\right)
\]
\normalsize
where the conjugacy class only depends on the residue of $a\in\mathcal{O}_{-5}$ modulo $(3)$. Moreover, as in the proof of Theorem~\ref{thm:fullreiner}, scaling by global units from $\mathcal{O}_{-5}[\zeta_3]$ is allowed. In particular, we get 5 conjugacy classes, using the representatives $a\in \{0,1,\sqrt{-5},1+\sqrt{-5},1-\sqrt{-5}\}$ of the $\mathcal{O}_{-5}[\zeta_3]^\times$-orbits on $\mathcal{O}_{-5}/(3)$. We can also see that conjugation by 
\footnotesize
\[
\left(\begin{array}{ccc} 0&1&a\\1&0&0\\0&0&1\end{array}\right)
\]
\normalsize
takes the square of a matrix as above (with (3,1)-entry $a$) to one with (3,1)-entry $2a$ which mod 3 is in the same global unit orbit as the original matrix. In particular, all the above matrices are conjugate to their squares, hence the corresponding cyclic groups embed in dihedral groups. 

The conjugacy relations between the matrices arising from the non-trivial ideal class are not so easy to identify. In particular, choosing suitable bases as before is significantly more difficult because we can only choose a basis in the full module $M_{\op{N}}\oplus M/M_{\op{N}}$ but not in the individual summands. A discussion as before for the trivial ideal class is therefore not possible. At this point, we do not know what the appropriate generalization of Theorem~\ref{thm:fullreiner} or Proposition~\ref{prop:halfreiner} should be, and how to properly identify the number of conjugacy classes of elements resp. subgroups. The computer calculations for the case $\mathcal{O}_{-5}$ suggest that there are six conjugacy classes of order 3 elements related to the non-trivial ideal class; and the Galois group of $\mathbb{Q}(\zeta_3,\sqrt{-5})/\mathbb{Q}(\sqrt{-5})$ fixes four of these classes (which hence acquire a dihedral overgroup) and interchanges the remaining two.

\subsubsection*{Acknowledgement} We would like to thank Oliver Br\"aunling for comments and discussions on various versions of the manuscript, the IRMA Strasbourg for hosting a long-term research stay of Bui Anh Tuan, and Gabor Wiese for providing us his computing nodes. We also would like to thank Eamonn O'Brien and an anonymous referee for their comments on the submitted version which led to several clarifications, and helped improve the presentation of the paper.


\begin{bibdiv}
\begin{biblist}

\end{biblist}
\end{bibdiv}


\begin{bibdiv}
 \begin{biblist}
\bib{Ash}{article}{
author={Ash, Avner},
title={Small-dimensional classifying spaces for arithmetic subgroups of general linear groups}, 
journal={Duke Math. J.},
volume={51},
date={1984},
pages={459--468},
}
\bib{bird:parry}{article}{
   author={Bird, R.H.},
   author={Parry, C.J.},
   title={Integral bases for bicyclic biquadratic fields over quadratic subfields},
   journal={Pacific J. Math.},
   volume={66},
   date={1976},
   number={1},
   pages={29--36}
}
\bib{magma}{article}{
    AUTHOR = {Bosma, Wieb},
    author = {Cannon, John},
    author ={Playoust, Catherine},
     TITLE = {The {M}agma algebra system. {I}. {T}he user language},
      NOTE = {Computational algebra and number theory (London, 1993)},
   JOURNAL = {J. Symbolic Comput.},
    VOLUME = {24},
      YEAR = {1997},
    NUMBER = {3-4},
     PAGES = {235--265},
      ISSN = {0747-7171},
  review = {MR1484478},
}
 \bib{BCNS}{article}{
    Author = {Braun, Oliver},
  author =  {{Coulangeon}, Renaud},
  author =  {Nebe, Gabriele},
  author =  {{Sch\"onnenbeck}, Sebastian},
    Title = {{Computing in arithmetic groups with Vorono\"{\i}'s algorithm}},
    Journal = {{J. Algebra}},
    Volume = {435},
    Pages = {263--285},
    Year = {2015},
   review={~Zbl 1323.16014}
}
\bib{Brown}{book}{
   author={Brown, Kenneth S.},
   title={Cohomology of groups},
   series={Graduate Texts in Mathematics},
   volume={87},
   note={Corrected reprint of the 1982 original},
   publisher={Springer-Verlag, New York},
   date={1994},
   pages={x+306},
   isbn={0-387-90688-6},
   review={\MR{1324339}},
}


\bib{Bui}{book}{
  author =  {Bui, Anh Tuan},
    author =  {Rahm, Alexander D.},
    title =   {Torsion Subcomplexes Subpackage, version 2.1},
  address = {accepted sub-package in HAP (Homological Algebra Programming)~\cite{HAP} in the computer algebra system GAP \cite{GAP4}},
  year =    {2018, Source code available at
  \url{http://gaati.org/~rahm/subpackage-documentation/} \qquad ${}$},
}

\bib{psl4z}{article}{
  author =  {Bui, Anh Tuan},
  author =  {Rahm, Alexander D.},
  author = {Wendt, Matthias},
  title =   {The Farrell--Tate and Bredon homology for $\op{PSL}_4(\mathbb{Z})$ via cell subdivisions},
  journal = {J. Pure Appl. Algebra},
  volume = {223},
  number = {7},
  pages = {2872--2888},
  note = {arXiv:1611.06099},
  year =    {2019}
}

\bib{github}{book}{
  author =  {Bui, Anh Tuan},
  author =  {Rahm, Alexander D.},
  author = {Wendt, Matthias},
  title =   {Source code and computational data used in the present paper},
  address={\\ \url{https://github.com/arahm/Farrell-Tate_cohomology_computations}},
  year =    {2022}
}

\bib{dutour:ellis:schuermann}{article}{
   author = { {Dutour~Sikiri\'c}, Mathieu },
  author =  {Ellis, Graham J.},
   author = { {Schuermann}, Achill },
   Title = {{On the integral homology of $\text{PSL}_4(\mathbb Z)$ and other arithmetic groups}},
    Journal = {{J. Number Theory}},
    Volume = {131},
    Number = {12},
    Pages = {2368--2375},
    Year = {2011},
   review={~Zbl 1255.11028}
}

\bib{DGGHSY}{article}{
   author = { {Dutour~Sikiri\'c}, Mathieu },
   author = {{Gangl}, Herbert},
   author = {{Gunnells}, Paul~E.},
   author = {{Hanke}, Jonathan },
   author = { {Schuermann}, Achill },
   author = { {Yasaki}, Dan},
   title = {On the cohomology of linear groups over imaginary quadratic fields},
   journal = {J. Pure Appl. Algebra},
   volume = {220},
   number = {7},
   pages = {2654--2589},
  note = {arXiv: 1307.1165},
     year = {2016}
}

\bib{HAP}{incollection}{
      author={Ellis, Graham},
       title={Homological algebra programming},
        date={2008},
   booktitle={Computational group theory and the theory of groups},
      series={Contemp. Math.},
      volume={470},
   publisher={Amer. Math. Soc.},
     address={Providence, RI},
       pages={63\ndash 74},
      review={\MR{2478414 (2009k:20001)}, \textit{implemented in the HAP package in the GAP computer algebra system \cite{GAP4}}},
}
\bib{GAP4}{book}{
  author={The GAP~Group}, 
  title={GAP -- Groups, Algorithms, and Programming, Version 4.9.3},
  year         = {2018},
  address   = {\\ \url{https://www.gap-system.org}},
    }
\bib{kubota}{article}{
   author={Kubota, T.},
   title={\"Uber den bizyklischen biquadratischen Zahlk\"orper},
   journal={Nagoya Math. J.},
   volume={10},
   date={1956},
   pages={65--85}
}

\bib{latimer:macduffee}{article}{
   author={Latimer, Claiborne G.},
   author={MacDuffee, C. C.},
   title={A correspondence between classes of ideals and classes of
   matrices},
   journal={Ann. of Math. (2)},
   volume={34},
   date={1933},
   number={2},
   pages={313--316},
   issn={0003-486X},
   review={\MR{1503108}},
   doi={10.2307/1968204},
}

\bib{neukirch}{book}{
 Author = {Neukirch, J{\"u}rgen},
 Title = {Algebraic number theory. {Transl}. from the {German} by {Norbert} {Schappacher}},
 Series = {Grundlehren Math. Wiss.},
 ISSN = {0072-7830},
 Volume = {322},
 ISBN = {3-540-65399-6},
 Year = {1999},
 Publisher = {Berlin: Springer},
 review={zbMATH 1313469, Zbl 0956.11021}
}
\bib{PariGP}{book}{
      author = {The PARI~Group},
      title        = {PARI/GP version \texttt{2.9.5}},
      year         = {2018},
      address      = {Univ. Bordeaux, available from \url{http://pari.math.u-bordeaux.fr/}}
}
\bib{accessingFarrell}{article}{
   author={Rahm, Alexander D.},
   title={Accessing the cohomology of discrete groups above their virtual cohomological dimension},
   journal={J. Algebra},
   volume={404},
   date={2014},
   pages={152--175},
   issn={0021-8693},
   review={\MR{3177890}},
}

\bib{sl2ff}{article}{
  author =  {Rahm, Alexander D.},
  author =  {Wendt, Matthias},
  title =   {On Farrell--Tate cohomology of $\op{SL}_2$ over $S$-integers},
  JOURNAL = {J. Algebra},
  VOLUME = {512},
  PAGES = {427--464},
  YEAR = {2018}
  }

\bib{reiner:1955}{article}{
   author={Reiner, Irving},
   title={Integral representations of cyclic groups of prime order},
   journal={Proc. Amer. Math. Soc.},
   volume={8},
   date={1957},
   pages={142--146},
   issn={0002-9939},
   review={\MR{0083493}},
}
\bib{Sage}{manual}{
      author={Developers, The~Sage},
       title={{S}agemath, the {S}age {M}athematics {S}oftware {S}ystem
  ({V}ersion 7.6)},
        date={2017},
        note={\\doi.org/10.5281/zenodo.820864 \qquad {\tt https://www.sagemath.org}},
}
\bib{Sebastian}{article}{
   author = {{Sch{\"o}nnenbeck}, Sebastian},
    title = {Resolutions for unit groups of orders},
journal={Journal of Homotopy and Related Structures},
year={2017},
volume={12},
number={4},
pages={837--852},
issn={1512-2891},
doi={10.1007/s40062-016-0167-6},
}

\end{biblist}
\end{bibdiv}
\end{document}